\documentclass[bj]{imsart}

\usepackage{amsmath,amsthm,amsfonts,amssymb}
\usepackage{dsfont,extarrows,enumerate}
\usepackage{hyperref}
\usepackage{graphicx}
\usepackage{bbm}
\RequirePackage[numbers]{natbib}

\startlocaldefs
\renewcommand{\P}{\mathbf P}
\newcommand{\Prob}[1]{\P\left\{#1\right\}}
\newcommand{\E}{\mathbf E}

\newcommand{\R}{\mathbb R}

\newcommand{\NN}{\mathbb N}

\newcommand{\bMh}{\widehat{\mathbf M}}

\newcommand{\bQh}{\widehat{\mathbf Q}}

\newcommand{\salg}{\mathfrak{F}}

\newcommand{\leb}{\mathrm{Leb}}

\newcommand{\sA}{\mathcal{A}}

\newcommand{\sG}{\mathcal{G}}
\newcommand{\sX}{\mathcal{X}}
\newcommand{\sP}{\mathcal{P}}

\newcommand{\Borel}{\mathcal{B}}

\newcommand{\one}{\mathbbm{1}}
\newcommand{\dsim}{\overset{d}{=}}

\newcommand{\fod}{\overset{{\rm f.d.}}{=}}

\newcommand{\toas}{\overset{{\rm a.s.}}{\longrightarrow}}

\DeclareMathOperator{\Id}{Id}
\DeclareMathOperator{\diff}{d}

\renewcommand{\epsilon}{\varepsilon}
\newcommand{\eps}{\epsilon}

\theoremstyle{plain} \newtheorem{theorem}{Theorem}[section]
\theoremstyle{plain} \newtheorem{proposition}[theorem]{Proposition}
\theoremstyle{plain} \newtheorem{lemma}[theorem]{Lemma}
\theoremstyle{plain} \newtheorem{corollary}[theorem]{Corollary}
\theoremstyle{definition} 
\theoremstyle{definition} 
\theoremstyle{definition} \newtheorem{conj}[theorem]{Conjecture}
\theoremstyle{remark} \newtheorem{remark}[theorem]{Remark}
\theoremstyle{remark} \newtheorem{example}[theorem]{Example}

\endlocaldefs

\begin{document}

\begin{frontmatter}

\title{Sieving random iterative function systems}
\runtitle{Sieving random iterative function systems}

\begin{aug}

\author{\fnms{Alexander} \snm{Marynych}\thanksref{a1,e1}\ead[label=e1,mark]{marynych@unicyb.kiev.ua}}
\and
\author{\fnms{Ilya} \snm{Molchanov}\thanksref{a2,e2}\ead[label=e2,mark]{ilya.molchanov@stat.unibe.ch}}

\address[a1]{Faculty of Computer Science and Cybernetics\\ 
Taras Shevchenko National University of Kyiv\\ 
Volodymyrska 60\\
01601 Kyiv, Ukraine\\
\printead{e1}}

\address[a2]{Institute of Mathematical Statistics\\
and Actuarial Science\\
University of Bern\\
Alpeneggstrasse 22\\
CH-3012 Bern, Switzerland\\
\printead{e2}}

\runauthor{A. Marynych and I. Molchanov}

\affiliation{Taras Shevchenko National University of Kyiv\thanksmark{m1} and University of Bern\thanksmark{m2}}

\end{aug}

\begin{abstract}
It is known that backward iterations of independent copies of a
  contractive random Lipschitz function converge almost surely under
  mild assumptions. By a sieving (or thinning) procedure based on
  adding to the functions time and space components, it is possible to
  construct a scale invariant stochastic process. We study its
  distribution and paths properties. In particular, we show that it is
  c\`adl\`ag and has finite total variation. We also provide examples
  and analyse various properties of particular sieved iterative
  function systems including perpetuities and infinite Bernoulli
  convolutions, iterations of maximum, and random continued
  fractions.
\end{abstract}



\end{frontmatter}

\section{Introduction}
\label{sec:introduction}

Iteration is one of fundamental tools in mathematics going back to
famous fixed point theorems for contractive mappings. In probabilistic
setting, one often works with iterated independent identically
distributed (i.i.d.) Lipschitz functions $(f_i)_{i\in\NN}$ defined on
a complete separable metric space and study the convergence of either
backward $f_1(f_2(f_3(\cdots f_n(\cdot))))$ or forward
$f_n(f_{n-1}(f_{n-2}(\cdots f_1(\cdot))))$ iterations as
$n\to\infty$. An incomplete list of early works on random iterations
include \cite{Chamayou+Letac:1991,deGroot+Rao:1963, Dubins+Freedman:1967,Duflo:1997,Letac:1986} and references therein. A comprehensive study of convergence regimes for
contractive (a precise definition will be given below) iterated random functions goes back to the prominent paper by Diaconis and Freedman \cite{diac:freed99}.

An important special case of iterated random affine mappings (called stochastic perpetuities) was studied in
\cite{Alsmeyer+Iksanov+Roesler:2008,gol:mal00,Grincevicius:1980,kes73}
and in many other works. The recent books \cite{bur:dam:mik16,Iksanov:2016} provide more comprehensive lists of further references. A particular instance of perpetuities, infinite Bernoulli convolutions, have been attracting enormous attention since 1930th,
see for example
\cite{Akiyama_et_al:2018,Erdoes:1939,Erdoes:1940,Peres+Schlag+Solomyak:2000,Solomyak:1995,Varju:2019}. 

We recall the main setting, restricting ourselves to the case of
Lipschitz functions on the real line $\R$. Let $\sG$ be the space of
Lipschitz functions $f:\R\mapsto\R$ endowed with the usual Lipschitz
norm $\|f\|_{\mathrm{Lip}}:=|f(0)|+L_f$, where
\begin{displaymath}
  L_f:=\sup_{x,y\in\R,x\neq y}\frac{|f(y)-f(x)|}{|x-y|}
\end{displaymath}
is the Lipschitz constant of $f\in\sG$.  The composition of functions
$f\circ g$ defined by $(f\circ g)(x):=f(g(x))$ for $x\in\R$ endows
$\sG$ with the semigroup structure and is continuous with respect to
$\|\cdot\|_{\mathrm{Lip}}$. 

Equip $\sG$ with a probability measure $\nu$ on the Borel
$\sigma$-algebra of $\sG$. Since the composition operation is
continuous, the composition of two $\sG$-measurable functions is again
$\sG$-measurable. If $f$ is a random function with distribution $\nu$
such that
\begin{equation}
  \label{eq:cond1}
  K_f:=\E L_{f}=\int_{\sG}L_{f}\diff \nu(f)<\infty,\quad
  \E \log L_{f}=\int_{\sG}\log L_{f}\diff \nu(f)<0,
\end{equation}
and
\begin{equation}
  \label{eq:cond2}
  \E |f(z_0)-z_0|=\int_{\sG}|f(z_0)-z_0|\diff \nu(f)<\infty
\end{equation}
for some $z_0\in\R$, then the sequence of backward iterations 
\begin{equation}
  \label{eq:11}
  Z_n:=f_{1}\circ\cdots\circ f_n(z_0)
\end{equation}
converges almost surely as $n\to\infty$ and the limit $Z_{\infty}$
does not depend on the choice of $z_0$, see Theorem~1 and
Proposition~1 in \cite{diac:freed99}. From this, one deduces that the
sequence of forward iterations $f_n \circ\cdots\circ f_1(z_0)$
converges in distribution to $Z_{\infty}$, see Theorems~1.1 and~5.1 in
\cite{diac:freed99}. Furthermore, the limiting random variable
$Z_{\infty}$ satisfies the stochastic fixed-point equation
\begin{equation}
  \label{eq:fixed-point}
  Z_{\infty}\dsim f(Z_{\infty}),
\end{equation}
where $f$ and $Z_{\infty}$ on the right-hand side are independent.

Many important distributions appear as limits for random iterated
functions. This work aims to extend this construction in order to come
up with stochastic processes (in general, set-indexed) whose
univariate distributions arise from iterations and joint distributions
are related by leaving some iterations out. For instance, assume that each 
of the functions $f_i$ is associated with
a uniformly distributed random variable $U_i$ and is deleted from the
iteration chain in \eqref{eq:11} if $U_i$ exceeds a given number
$x$. The limit of such iterations is a random variable $\zeta(x)$
whose distribution is the same as that of $Z_\infty$. The properties
of $\zeta(x)$ considered a random function of $x$ is the main subject
of this paper.

As a preparation to a general construction of such stochastic
processes presented in Section~\ref{sec:iterated-functions} we shall provide a few
examples.

\begin{example}\label{ex:bernoulli_convolutions}
  Consider an infinite sequence $(Q_n)_{n\in\NN}$ of independent
  copies of a random variable $Q$ taking values $0$ or $1$ equally
  likely. For $\lambda\in(0,1)$, the Bernoulli convolution
  \begin{displaymath}
    Z_\infty:=\sum_{n=1}^\infty \lambda^{n-1} Q_n
  \end{displaymath}
  results from the backward iteration of independent copies of the
  function $f(z)=\lambda z+Q$. Now consider a sequence
  $(U_n)_{n\in\NN}$ of i.i.d.~uniform random variables on $[0, 1]$ which
  is independent of $(Q_n)_{n\in\NN}$. Put
  $T_k(x):=\sum_{j=1}^{k}\one_{\{U_j\leq x\}}$, $k\in\NN$,
  $x\in(0,1]$, where $\one_{\{x_i\in A\}}$ is the indicator of the
  event $\{x_i\in A\}$, and further
  $S_n(x):=\inf\{k\in\NN:T_k(x)=n\}$, $n\in\NN$, $x\in(0,1]$. Let
  \begin{displaymath}
    \zeta(x):=\sum_{n=1}^\infty \lambda^{n-1} Q_{S_n(x)},\quad x\in(0,1].
  \end{displaymath}
  This yields a stochastic process, whose univariate marginals are all
  distributed like $Z_\infty$. We will explore path properties of this
  process, show that for $\lambda\in(0,1/2]$ it is Markov in both
  forward and reverse time and find its generating operator. It is well
  known that if $\lambda=1/2$, then $\zeta(x)$ is uniformly
  distributed on $[0,2]$ for every $x\in(0,1]$. We show that the
  bivariate distributions are singular for some $x$ close enough to $1$,
  determine a bound on their Hausdorff dimension and calculate the local
  dimension on the set of binary rational points.
\end{example}

\begin{example}
  \label{ex:perpetuities}
  Generalising the previous example, consider a sequence
  $(Z_n)_{n\in\NN}$ of backward iterations of affine mappings
  $f_n(x)=M_n x+Q_n$, $n\in\NN$, applied to the initial point $z_0=0$,
  where $(M_n,Q_n)_{n\in\NN}$ are i.i.d.~random vectors in $\R^2$. A
  criterion for a.s.~convergence of $(Z_{n})$ is known, see
  \cite[Th.~2.1]{gol:mal00}. In particular, by
  \cite[Cor.~4.1]{gol:mal00} convergence takes place whenever
  $\E\log|M|\in(-\infty,0)$, $\E\log^+|Q|<\infty$, where $\log^{+}x:=\log(x\vee 1)$, and an additional
  nondegeneracy assumption, see formula \eqref{eq:perp_non_degenerate}
  below, holds.  Let $(U_n)_{n\in\NN}$ be i.i.d.~uniform random variables on $[0, 1]$ 
  which are independent of $(M_n,Q_n)_{n\in\NN}$. Consider a coupled family of processes
  \begin{displaymath}
    \zeta(x):=\sum_{n=1}^{\infty} \left(\prod_{k=1}^{n-1}
      M_k^{\one_{\{U_k\leq x\}}}\right)Q_n\one_{\{U_n\leq x\}}, \quad x\in(0,1].
  \end{displaymath}
  We establish the uniform convergence of partial sums of the above
  series to the limit $\zeta(x)$ and explore its path properties.
\end{example}

\begin{example}
  Consider a continued fraction $W_{n}=\frac{1}{W_{n-1}+a_{n}}$ with
  (possibly i.i.d.~random) coefficients $a_n>0$, $n\in\NN$. If $\sum
  a_n=\infty$ a.s., the continued fraction converges in distribution by the Stern--Stolz theorem, see Theorem 10 in 
  \cite{khinchincontinued}. Given once again a sequence $(U_n)_{n\in\NN}$
  of i.i.d.~uniform random variables on $[0, 1]$ which is independent of
  $(a_n)_{n\in\NN}$, we modify the continued fraction by letting
  \begin{displaymath}
    W_{n}(x)=\begin{cases}
      \frac{1}{W_{n-1}(x)+a_n}, &\text{if } U_n\leq x,\\
      W_{n-1}(x), &\text{if } U_n>x,
    \end{cases}
    \quad x\in(0,1].
  \end{displaymath}
  Note that for every fixed $x\in (0,1]$, $W_n(x)$ is the forward
  iteration of the mappings
  $$
  f_{n,x}(z)=\frac{1}{a_n+z}\one_{\{U_n\leq x\}}+z\one_{\{U_n>x\}},\quad z>0.
  $$
  The a.s.~pointwise limits of the corresponding backward iterations
  $Z_n(x)$ as $n\to\infty$ constitute a stochastic process on
  $(0,1]$. We show that this process has a finite total variation, and
  is Markov if $(a_n)_{n\in\NN}\subset \NN$. 
\end{example}

Note that in all above examples we eliminate some iterations from the
infinite sequence
$$
Z_{\infty}=f_1\circ f_2\circ\cdots\circ f_n\circ\cdots
$$
by replacing the corresponding functions with the identity mapping in
a coupled manner. In Section~\ref{sec:iterated-functions}, we suggest
a sieving scheme for iterated functions, which is generated by an
auxiliary Poisson point process. As a result, we are led to a set-indexed stochastic process whose univariate marginals are all the same and are distributed as the almost sure limit of $Z_n$ in
\eqref{eq:11}. By taking its values on the segments $[x,1]$ with
$x\in(0,1]$, we obtain all constructions mentioned in the above
examples as special cases. 

The distributional properties of the set-indexed process are analysed
in Section~\ref{sec:distr-prop-zetacd}, in particular, it is shown
that a variant of this process on the half-line is scale invariant. By restricting the process to a finite interval, it is possible to
rephrase our sieving scheme as iteration in a functional space. With
this idea, in Section~\ref{sec:iter-finite-interv} we use tools from
the theory of empirical processes to establish the uniform convergence on
some classes of sets. In Section~\ref{sec:cont-prop-zeta} it is shown
that the limiting process $\zeta$ is c\`adl\`ag and has a finite total
variation on any bounded interval separated from zero. We also discuss
integration with respect to $\zeta$ and integrability properties of
$\zeta$. Section~\ref{sec:markov-property} establishes the
Markov property of the process, assuming a kind of a strong
separation condition known in fractal geometry.

The most well-studied family of iterations are perpetuities, also
known as autoregressive processes of the first order, see
Example~\ref{ex:perpetuities} above. The sieving scheme is applied to
them in Section~\ref{sec:perpetuities}, where also an important
example of Bernoulli convolutions is considered, see
Example~\ref{ex:bernoulli_convolutions}. Section~\ref{sec:other-examples}
outlines several other instances of iterations that provide new
examples of self-similar stochastic processes. In the Appendix we collect some 
technical proofs and provide several further examples related to perpetuities.

\section{Sieving scheme for iterated functions}
\label{sec:iterated-functions}

Let $\sX$ be a complete separable metric space with its Borel
$\sigma$-algebra $\Borel(\sX)$ and equipped with a $\sigma$-finite
measure $\mu$. Recall that $\sG$ is the family of Lipschitz functions
on the real line with a probability measure $\nu$ satisfying
  \eqref{eq:cond1} and \eqref{eq:cond2}.

Let $\R_+:=[0,\infty)$ be the positive half-line with the Lebesgue
measure $\leb$. Consider a Poisson process $\sP$ on
$\R_+\times\sX\times\sG$ with intensity measure
$\leb\otimes\mu\otimes\nu$. Note that in a triplet $(t,x,f)\in\sP$
the function $f$ may be considered as a mark of the point $(t,x)$, the
marks of different points are independent and $\nu$ is the probability
distribution of the typical mark denoted by $f$. For
$A\in\Borel(\sX)$, denote by $\sP_A$ the intersection of $\sP$ with
$\R_+\times A\times\sG$.

For a sequence $(f_n)_{n\in\NN}$ of i.i.d.~random Lipschitz functions,
write
\begin{displaymath}
  f^{k\uparrow n}=f_k\circ\cdots\circ f_n,
\end{displaymath}
and $f^{k\uparrow \infty}$ for the almost sure limit of these
iterations as $n\to\infty$ provided it exists.  For $k>n$ we stipulate
that $f^{k\uparrow n}$ is the identity function $\Id$.

For each $A\in\Borel(\sX)$ with $\mu(A)\in(0,\infty)$, enumerate the
points $\{(t_{k,A},x_{k,A},f_{k,A}):\; x_{k,A}\in A,\; k\geq1\}$ of $\sP_A$,
so that the first component is a.s.~increasing, and define the
\emph{sieved backward iterations} of $(f_{k,A})_{k\in\NN}$:
\begin{equation}
  \label{eq:1}
  \zeta_t(A):=f_A^{1\uparrow N_A(t)}(z_0)=f_{1,A}\circ\cdots \circ f_{N_A(t),A}(z_0),\quad t > 0,
\end{equation}
where $z_0\in\R$ is fixed and nonrandom, and
\begin{displaymath}
  N_A(t):=\sup\{k\geq 1:t_{k,A}\leq t,x_{k,A}\in A\}
\end{displaymath}
with the convention $\sup\varnothing=0$. Thus, $\zeta_t(A)$ is a
finite backward composition of marks $f_k$ for $(t_k,x_k)$ from the
rectangle $[0,t]\times A$.  Equivalently, $\zeta_t(A)$ is the
composition of functions $f_i\one_{\{x_i\in A\}}+\Id \one_{\{x_i\notin
  A\}}$ for $t_i\leq t$ applied to the starting point $z_0$. This equivalent interpretation makes transparent the ``sieved'' structure of $\zeta_t(A)$ which might seem disguised in the definition \eqref{eq:1}.

In what follows we always assume that conditions \eqref{eq:cond1} and
\eqref{eq:cond2} are in force.  Then $\zeta_t(A)$ in (\ref{eq:1})
converges almost surely as $t\to\infty$. The limiting random element
is denoted by $\zeta(A)$; it is a random set-indexed function
defined on
\begin{displaymath}
  \Borel_{+}(\sX):=\{A\in\Borel(\sX):\;\mu(A)\in(0,\infty)\}.
\end{displaymath}
Furthermore, $\zeta(A)$ does not depend on the choice of $z_0$. 

If $\sX$ is the half-line $\R_+=[0,\infty)$ and $\mu$ is the Lebesgue
measure, then we work with a Poisson process on
$\R_+\times\R_+\times\sG$, and, for $A=[0,x]$ with $x>0$, the random
variable $\zeta_t(A)$ is the result of iterating the functions $f_i$
ordered according to $t_i\leq t$ and such that $x_i\leq x$. In this case, we
write $\zeta(x)$ as a shorthand for $\zeta([0,x])$, $x > 0$, and
regard $(\zeta(x))_{x > 0}$ as a stochastic process on
$(0,\infty)$. Note that by passing from $\zeta(x)$ to $\zeta(y)$, we sieve some iterations out
if $y<x$ and insert additional ones if $y>x$.

\section{Distributional properties}
\label{sec:distr-prop-zetacd}

\subsection{Finite-dimensional distributions and scale invariance}

Recall that $\zeta(A)$ is defined for $A\in\Borel_+(\sX)$, that is,
for Borel $A$ such that $\mu(A)\in(0,\infty)$.  Note the following
simple facts.

\begin{proposition}
  \label{prop:same-distr}
  The distribution of $\zeta(A)$ does not depend on
  $A\in\Borel_+(\sX)$ and $\zeta(A)\dsim Z_{\infty}$. If $\P\{Z_{\infty}=0\}<1$, the set
  function $\zeta$ is not additive, and hence is not a measure on
  $\sX$. If $A_1\cap A_2=\varnothing$ for $A_1,A_2\in\Borel_+(\sX)$,
  then $\zeta(A_1)$ and $\zeta(A_2)$ are independent.
\end{proposition}

\begin{theorem}
  \label{thr:s-sim}
  Let $\phi:\sX\mapsto\sX$ be any measurable bijection such
  that $\mu(\phi^{-1}(A))=c\mu(A)$ for a constant $c>0$ and all
  $A\in\Borel_+(\sX)$. Then $\zeta(\phi(A))$ and $\zeta(A)$ share the
  same finite-dimensional distributions as functions of
  $A\in\Borel_+(\sX)$.
\end{theorem}
\begin{proof}
  By the transformation theorem for Poisson processes, the process
  with intensity measure $\mu(\phi^{-1}(A))$ can be obtained as
  $(\phi(x_i))_{i\in\NN}$, where $(x_i)_{i\in\NN}$ is the Poisson
  process with intensity $\mu$. Thus, $\zeta(\phi(A))$,
  $A\in\Borel_+(\sX)$, coincides with the limiting set-indexed process
  obtained by using Poisson points from the process of intensity $c
  \leb\otimes \mu\otimes\nu$. This process is obtained from the
  original one by transform $t_i\mapsto c^{-1}t_i$, which does not
  change the order of the $t_i$s and so the limit in \eqref{eq:1}.
\end{proof}

It is possible to describe two-dimensional distributions of the
set function $\zeta$ as follows. Let $A_1,A_2$ be two sets from
$\Borel_+(\sX)$. Consider the triplet $(t_*,x_*,f_*)$ such that $t_*$
is the smallest among all triplets $(t_i,x_i,f_i)$ with $x_i\in
A_1\cup A_2$. Then
\begin{multline}
  \label{eq:two-dim-fixed-point}
  (\zeta(A_1),\zeta(A_2))
  \dsim (f_*(\zeta(A_1)),f_*(\zeta(A_2)))\one_{\{x_*\in A_1\cap A_2\}}\\
  + (\zeta(A_1),f_*(\zeta(A_2)))\one_{\{x_*\in A_2\setminus A_1\}} + 
  (f_*(\zeta(A_1)),\zeta(A_2))\one_{\{x_*\in A_1\setminus A_2\}}.
\end{multline}
A similar equation can be written for the joint distribution of
$(\zeta(A_1),\zeta(A_2),\dots,\zeta(A_m))$ for any
$A_1,\dots,A_m\in\Borel_+(\sX)$.

In the special case of $\sX=\R_+$ with $\mu$ being the Lebesgue
measure, Theorem~\ref{thr:s-sim} yields that the finite-dimensional
distributions of $(\zeta(x))_{x>0}$ do not change after scaling of its
argument by any positive constant, meaning that $(\zeta(x))_{x>0}$ is
scale invariant.  After the exponential change of time, the process
$\tilde{\zeta}(s):=\zeta(e^s)$, $s\in\R$, is strictly stationary on
$\R$.

\subsection{Power moments}
\label{sec:power-moments}
 
Using known results for perpetuities, it is easy to deduce the
following statement.

\begin{proposition}
  \label{prop:moments}
  Assume that, for some $p>0$, we have $\E L_f^p<1$ and $\E
  |f(z_0)-z_0|^p <\infty$. Then $\E |\zeta(A)|^p<\infty$ for all
  $A\in\Borel_{+}(\sX)$.
\end{proposition}
\begin{proof}
  For every $A\in\Borel_{+}(\sX)$, the random variable $\zeta(A)$ has
  the same distribution as $Z_{\infty}$. By the triangle inequality
  \begin{equation}
    \label{eq:values_of_f_is_in_cone}
    |f_i(z)-z_0|\leq |f_i(z_0)-z_0|+L_{f_i}|z-z_0|,\quad z\in\R,
  \end{equation}
  and, therefore,
  \begin{displaymath}
    |f^{1\uparrow n}(z_0)-z_0|
    \leq \sum_{k=1}^{n}|f_k(z_0)-z_0|\prod_{j=1}^{k-1}L_{f_j}\quad \text{a.s.},\quad n\in\NN.
  \end{displaymath}
  Letting $n\to\infty$ yields 
  \begin{displaymath}
    |Z_{\infty}-z_0|^p\leq \left(\sum_{k=1}^{\infty}|f_k(z_0)-z_0|
      \prod_{j=1}^{k-1}L_{f_j}\right)^p \quad \text{a.s.}
  \end{displaymath}
  The term in the parentheses on the right-hand side is a
  perpetuity. The criterion for existence of power moments of
  perpetuities is given in 
  \cite[Th.~1.4]{Alsmeyer+Iksanov+Roesler:2008}. In particular, under our
  assumptions the right-hand side of the last display is finite. The
  proof is complete.
\end{proof}

\begin{remark}
The inequality $\E |Z_{\infty}|^p<\infty$ is stated under weaker assumptions in Theorem 2.3(d) in \cite{Alsmsyer+Fuh:2001}. However, in the claimed generality this result does not hold, see the corrigendum \cite{Alsmsyer+Fuh:2002} for a correct form which is weaker than Proposition \ref{prop:moments}.
\end{remark}

If the conditions of Proposition~\ref{prop:moments} hold for $p=2$,
then $\zeta(A)$ is square integrable for all $A\in\Borel_+(\sX)$, and 
\eqref{eq:two-dim-fixed-point} leads to an iterative equation for the
second moments of $\zeta$ as
\begin{multline}
  \label{eq:3}
  \mu(A_1\cup A_2)\E(\zeta(A_1)\zeta(A_2))
  =\mu(A_1\cap A_2)\E(f(\zeta(A_1))f(\zeta(A_2)))\\
	+\mu(A_1\setminus A_2)\E(f(\zeta(A_1))(\zeta(A_2)))
  +\mu(A_2\setminus A_1)\E(\zeta(A_1)f(\zeta(A_2))),
\end{multline}
where $f$ is a random element in $\sG$ with distribution $\nu$
independent of $\zeta(A_1)$ and $\zeta(A_2)$. For processes on the
half-line, \eqref{eq:3} becomes
\begin{displaymath}
  y\E(\zeta(x)\zeta(y))
  =x \E (f(\zeta(x))f(\zeta(y)))+(y-x)\E (\zeta(x)f(\zeta(y))),
  \quad 0<x\leq y,
\end{displaymath}
where $f$ is independent of $\zeta(x)$ and $\zeta(y)$.

\section{Iterations in a finite measure space}
\label{sec:iter-finite-interv}

\subsection{Iterations in a functional space}
\label{sec:iter-funct-space}

Assume that $\mu$ is not identically zero and finite on $\sX$, that is,
$\mu(\sX)\in(0,\infty)$. Then, the construction of the limiting
process can be done as follows. Let $(f_i)_{i\in\NN}$ be a
sequence of i.i.d. copies of $f$ from $\sG$ distributed according to
$\nu$, and let $(U_i)_{i\in\NN}$ be i.i.d. copies of a random element
$U\in\sX$ with distribution
\begin{equation}
  \label{eq:distribution_of_U}
  \P\{U\in A \}=\frac{\mu(A)}{\mu(\sX)},\quad A\in\Borel(\sX).
\end{equation}
Assume further that $(f_i)_{i\in\NN}$ and $(U_i)_{i\in\NN}$ are
independent. 

Let $A\in\Borel_+(\sX)$. Then $f_i$ contributes to the iterations
constituting $\zeta(A)$ if $U_i\in A$, and otherwise $f_i$ is replaced
by the identity map. In other words, we have the following identity
\begin{displaymath}
  \zeta(A)=f^{1\uparrow\infty}_A(z_0), \quad A\in\Borel_{+}(\sX),
\end{displaymath}
where the limit $f^{1\uparrow\infty}_A(z_0)$ of $f^{1\uparrow
  n}_A(z_0)$, as $n\to\infty$, is understood in the a.s.~sense,
$z_0\in\R$, and $f^{1\uparrow n}_A$ are backward iterations of
i.i.d.~copies of the function
\begin{equation}
  \label{eq:f_i_A_definition}
  f_{A}(\cdot):=f(\cdot)\one_{\{U\in A\}}+\Id(\cdot)\one_{\{U\notin A\}}.
\end{equation}
The set function
$(\zeta(A))_{A\in\Borel_{+}(\sX)}$ is the solution of the following
iterative distributional equation
\begin{equation}
  \label{eq:2}
  (\zeta(A))_{A\in\Borel_{+}(\sX)}\fod 
  (f(\zeta(A))\one_{\{U\in A\}}+\zeta(A)\one_{\{U\notin A\}})_{A\in\Borel_{+}(\sX)},
\end{equation}
where $f$, $U$ and $\zeta$ on the right-hand side are
independent. Note that the Lipschitz constant of $f_{A}$ is 
\begin{displaymath}
  L_{f_{A}}=L_{f}\one_{\{U\in A\}}+\one_{\{U\notin A\}},
\end{displaymath}
hence,
\begin{displaymath}
  \log L_{f_{A}}=(\log L_{f})\one_{\{U\in A\}}. 
\end{displaymath}

\begin{example}
  Assume that $\sX=[0,1]$ with the Lebesgue measure. Then $U$
  has the standard uniform distribution on $[0,1]$, and
  \begin{displaymath}
   \zeta(x)=f^{1\uparrow \infty}_x(z_0),\quad x\in (0,1],
  \end{displaymath}
  where the a.s.~limit does not depend on $z_0\in\R$, and
  $f^{1\uparrow \infty}_x$ are iterations composed of i.i.d.~copies of the
  function
  \begin{displaymath}
    f_{x}(\cdot):=f(\cdot)\one_{\{U\leq x\}}
    +\Id(\cdot)\one_{\{U> x\}}.
  \end{displaymath}
\end{example}

\subsection{Uniform convergence of sieved iterations}
\label{sec:uniform-convergence}

We now aim to prove the uniform convergence of iterations as functions
of Borel set $A$ by reducing the problem to the uniform convergence of
empirical processes. Let $\sA$ be a subclass of Borel sets in $\sX$. A
finite set $I$ of cardinality $n$ is shattered by $\sA$ if each of its
$2^n$ subsets can be obtained as $I\cap A$ for some $A\in\sA$. The
Vapnik--\v{C}ervonenkis  dimension of $\sA$ is the supremum of
cardinalities $n$ of all finite sets $I$ in $\sX$ shattered by
$\sA$. The family $\sA$ is called a Vapnik--\v{C}ervonenkis class if its
Vapnik--\v{C}ervonenkis dimension is finite. We refer to the classical book \cite{Vapnik:1998}, see in particular Section 4.9 therein, for the details of the Vapnik--\v{C}ervonenkis theory.

\begin{theorem}
  \label{thr:uniform}
  Let $\sA$ be a collection of Borel subsets of $\sX$ with
  $\mu(\sX)<\infty$ such that $\sA$ is a Vapnik--\v{C}ervonenkis
  class and $\inf_{A\in\sA}\mu(A)>0$.  Then
  \begin{equation}
    \label{eq:12}
    \sup_{A\in\sA}\left|\zeta(A)-f^{1\uparrow n}_{A}(z_0)\right|\toas  0 
    \quad \text{as }\; n\to\infty.
  \end{equation}
\end{theorem}

\begin{lemma}
  \label{lem:glivenko-cantelli-bis}
  Assume that a family $\sA$ satisfies conditions of
  Theorem~\ref{thr:uniform}. Further, let $(\xi_k)_{k\in\NN}$ be a
  sequence of i.i.d.~copies of an integrable random variable $\xi$
  such that $\E\xi<0$, and let $(U_k)_{k\in\NN}$ be a sequence of
  i.i.d.~copies of the random element $U$ with distribution
  \eqref{eq:distribution_of_U}, which is also independent of
  $(\xi_k)_{k\in\NN}$. Then
  \begin{equation}
    \label{eq:6}
    \sup_{A\in\sA}\frac{\sum_{k=1}^{n}\xi_k\one_{\{U_k\in A\}}}{n}
    \toas \frac{\inf_{A\in\sA}\mu(A)}{\mu(\sX)}\E\xi<0
    \quad \text{as }\;n\to\infty.
  \end{equation}
\end{lemma}
\begin{proof}
  Define a random measure (or abstract empirical process) on
  $\Borel(\sX)$ by
  \begin{displaymath}
    S_n(A):=\sum_{k=1}^{n}\phi_{nk}(A), \quad A\in\Borel(\sX),
  \end{displaymath}
  where 
  \begin{displaymath}
    \phi_{nk}(A):= \frac{1}{n}\xi_k\one_{\{U_k\in A\}},\quad
    k=1,\dots,n,\; A\in\Borel(\sX).
  \end{displaymath}
  Note that $\E S_n(A)=\mu(A)\E\xi/\mu(\sX)$ does not depend on
  $n$. For a function $\psi:\sA\mapsto\R$, denote
  $\|\psi\|:=\sup_{A\in\sA}|\psi(A)|$.

  Assume first that $|\xi_i|\leq c$ a.s. Then
  \begin{displaymath}
    \E \left(\sum_{k=1}^n \|\phi_{nk}\|\right)\leq c, \quad n\geq1.
  \end{displaymath}
  For $\delta>0$, 
  \begin{displaymath}
    \E \left(\sum_{k=1}^n \one_{\{\|\phi_{nk}\|>\delta\}}\|\phi_{nk}\|\right)
    \leq \frac{c}{n} \E \left(\sum_{k=1}^n \one_{\{|\xi_k|>n\delta\}}\right)
    =c\Prob{|\xi_k|>n\delta}\to 0 \quad \text{as }\; n\to\infty.
  \end{displaymath}
  By results from the theory of empirical processes (see
  \cite[Ch.~6]{duem17}),
  \begin{equation}
    \label{eq:14}
    \E \|S_n-\E S_n\|\to 0 \quad \text{as }\; n\to\infty.
  \end{equation}
  The sequence $S_n(A)-\E S_n(A)$ is a reverse martingale for each
  $A$, and so $\| S_n-\E S_n\|$ is a reverse submartingale which is
  bounded, since $|S_n(A)-\E S_n(A)|\leq 2c$ for all $A$ and $n$. Thus, 
  \begin{displaymath}
    \|S_n-\E S_n\|\toas 0 \quad\text{as }\; n\to\infty. 
  \end{displaymath}
  For a not necessarily bounded $\xi_k$, decompose the random measure
  as
  \begin{displaymath}
    S_n(A)=S'_n(A)+S''_n(A)=\sum_{k=1}^n
    \phi_{nk}(A)\one_{\{|\xi_k|\leq c\}} +\sum_{k=1}^n
    \phi_{nk}(A)\one_{\{|\xi_k|>c\}}.
  \end{displaymath}
  Then $\|S'_n-\E S'_n\|\to 0$ a.s. as $n\to \infty$ by the argument
  above applicable to the bounded $\xi_k$s. Furthermore,
  \begin{displaymath}
    \|S''_n-\E S''_n\|\leq \frac{1}{n}\sum_{k=1}^n 
    |\xi_k| \one_{\{|\xi_k|>c\}}+ \frac{1}{n}\sum_{k=1}^n
    \E(|\xi_k| \one_{\{|\xi_k|>c\}}). 
  \end{displaymath}
  The a.s. upper limit of the right-hand side can be made arbitrarily
  small by the choice of $c$.  Therefore,
  \begin{displaymath}
    \sup_{A\in\sA}\frac{\sum_{k=1}^{n}\xi_k
      \one_{\{U_k\in A\}}}{n}\toas 
    \sup_{A\in\sA}  \frac{\mu(A)}{\mu(\sX)}\E\xi
    =\frac{\inf_{A\in\sA}\mu(A)}{\mu(\sX)}\E\xi\quad
    \text{as }\;n\to\infty,
  \end{displaymath}
  because $\E \xi<0$. 
\end{proof}

\begin{remark}
  It is possible to impose weaker conditions that guarantee the
  uniform convergence of $S_n(A)$ over $A\in\sA$. Let
  $N(\eps,\sA,\hat\rho_n)$ be the cardinality of the smallest
  $\eps$-net in $\sA$ with respect to the random pseudometric
  \begin{displaymath}
    \hat\rho_n(A,A'):=\sum_{k=1}^n |\phi_{nk}(A)-\phi_{nk}(A')|
    =\frac{1}{n} \sum_{k=1}^n |\xi_k|\one_{\{U_k\in A\triangle A'\}}.
  \end{displaymath}
  Then the Vapnik--\v{C}ervonenkis class assumption may be replaced by
  the assumption that $N(\eps,\sA,\hat\rho_n)$ converges to zero in
  probability as $n\to\infty$. A comprehensive treaty of uniform
  convergence for empirical processes can be
  found 
  in \cite[Ch.~3]{Vapnik:1998}. In particular, a necessary and
  sufficient conditions are given in Theorem 3.5 on p.~101 therein.
\end{remark}

\begin{proof}[Proof of Theorem~\ref{thr:uniform}]
  Let $(L_{f_{i,A}})_{i\in\NN}$ be the Lipschitz constants of the
  i.i.d.~copies $(f_{i,A})_{i\in\NN}$ of the function $f_A$ given by
  \eqref{eq:f_i_A_definition}.  By the Lipschitz property,
  \begin{equation}
    \label{eq:8}
    \left|\zeta(A)-f^{1\uparrow n}_{A}(z_0)\right|
    \leq L_{f_{1,A}}\cdots L_{f_{n,A}}|f^{(n+1)\uparrow\infty}_A(z_0)-z_0|
  \end{equation}
  for any $A\in\Borel_+(\sX)$. Moreover, for each $i\in\NN$ and an
  arbitrary $z\in\R$, we have, similarly to
  \eqref{eq:values_of_f_is_in_cone},
  \begin{align*}
    |f_{i,A}(z)-z_0|&=\left|z\one_{\{U_i\notin A\}}+f_i(z)\one_{\{U_i\in A\}}-z_0\right|\\
    &=\left|(z-z_0)\one_{\{U_i\notin A\}}-z_0\one_{\{U_i\in A\}}
      +f_i(z_0)\one_{\{U_i\in A\}}+(f_i(z)-f_i(z_0))\one_{\{U_i\in A\}}\right|\\
    &\leq |f_i(z_0)-z_0|+\left(L_{f_i}\one_{\{U_i\in A\}}
      +\one_{\{U_i\notin A\}}\right)|z-z_0|=:Q_i+L_{f_{i,A}}|z-z_0|.
  \end{align*}
  Iterating the above inequality for $|f_{i,A}(z)-z_0|$ yields
  \begin{equation}
    \label{eq:key_tail_estimate}
    |f^{(n+1)\uparrow\infty}_A(z_0) -z_0|
    \leq\sum_{i=n+1}^{\infty}Q_i\prod_{k=n+1}^{i-1}L_{f_{k,A}}.
  \end{equation}
  Plugging this upper bound into \eqref{eq:8}, we obtain
  \begin{equation}
    \label{eq:key_tail_estimate2}
    \left|\zeta(A)-f^{1\uparrow n}_{A}(z_0)\right|
    \leq \sum_{i=n+1}^{\infty}Q_i\prod_{k=1}^{i-1}L_{f_{k,A}} 
    \leq \sum_{i=n+1}^{\infty}Q_i\left(\sup_{A\in\sA}\prod_{k=1}^{i-1}L_{f_{k,A}}\right).
  \end{equation}
  To show that the right-hand side of the above inequality converges
  to zero a.s., it suffices to check that
  \begin{equation}
    \label{eq:10}
    \sum_{i=1}^{\infty}Q_i\left(\sup_{A\in\sA}
      \prod_{k=1}^{i-1}L_{f_{k,A}}\right)<\infty\quad\text{a.s.}
  \end{equation}
  This follows from Cauchy's radical test. Indeed,
  \begin{multline*}
    \limsup_{i\to\infty}\frac{1}{i}
    \Big(\log Q_i+ \sup_{A\in\sA}\sum_{k=1}^{i-1} \log L_{f_{k,A}}\Big)\\
    \leq \limsup_{i\to\infty}\frac{1}{i}\log^{+} Q_i
    + \limsup_{i\to\infty}\frac{1}{i}\sup_{A\in\sA}\sum_{k=1}^{i-1} 
      (\log L_{f_k})\one_{\{U_k\in A\}}
    =\frac{\inf_{A\in\sA}\mu(A)}{\mu(\sX)}\E\log L_{f}<0,
  \end{multline*}
  where the second $\limsup$ was calculated in
  Lemma~\ref{lem:glivenko-cantelli-bis}. The first $\limsup$ is equal to zero
  by the Borel--Cantelli lemma and the fact that $\E \log^{+} Q_1= \E \log^{+}
  |f_1(z_0)-z_0|<\infty$ by the assumption \eqref{eq:cond2}. Thus,
  \begin{equation}
    \limsup_{i\to\infty}\left(Q_i\left(\sup_{A\in\sA}
        \prod_{k=1}^{i-1}L_{f_{k,A}}\right)\right)^{1/i}
    \leq \exp\left\{\frac{\inf_{A\in\sA}\mu(A)}{\mu(\sX)}\E\log L_{f}\right\}<1,
  \end{equation}
  and this completes the proof.
\end{proof}

Since the Vapnik--\v{C}ervonenkis dimension of a monotone family of
sets is $2$, we obtain the following result.

\begin{corollary}
  \label{cor:monotone}
  Let $\sA=\{A_t,t\geq 0\}$ be a nondecreasing (respectively,
  nonincreasing) subfamily of $\Borel_+(\sX)$ of finite measure such
  that $\cup_t A_t$ (respectively, $\cap_t A_t$) is of finite positive
  measure. Then \eqref{eq:12} holds.
\end{corollary}

\subsection{Uniform convergence for sieved iterations on the
  half-line}
\label{sec:unif-conv-proc}

Let $\sX=\R_+$ be the half-line with $\mu$ being the Lebesgue measure.
Let us consider stochastic process
$(\zeta(x))_{x>0}=(\zeta([0,x]))_{x>0}$. By
Corollary~\ref{cor:monotone}, the iterations
\begin{equation}
  \label{eq:19}
  \zeta_n(x):=f^{1\uparrow n}_{[0,x]}(z_0), \quad n\geq1,
\end{equation}
converge a.s.~to $\zeta$ uniformly over $x\in[a,b]$ for each $0<a\leq
b<\infty$.  The following result establishes their uniform convergence
in $L^p$.

Denote $\Phi(x):=\E L_f^x$, and let
\begin{equation}
  \label{eq:i_def}
  \mathcal{I}:=\{x>0: \Phi(x)<1\}.
\end{equation}
The set $\mathcal{I}$ is not empty under assumption \eqref{eq:cond1},
because it contains all sufficiently small positive numbers. This
follows from the following three relations: $\Phi(0)=1$,
$\Phi'(0)=\E\log L_f<0$ and $\Phi(1)<\infty$.

\begin{proposition}
  \label{prop:L1-uniform}
  For each $a>0$ and $p\in \mathcal{I}\cap(0,1]$,  
  \begin{equation}
    \label{eq:12a}
    \E \sup_{x\in[a,1]}\left|\zeta(x)-\zeta_n(x)\right|^p\to 0
    \quad \text{as }\; n\to\infty.
  \end{equation}
\end{proposition}
\begin{proof}
  Repeat the arguments from the proof of Theorem~\ref{thr:uniform},
  see \eqref{eq:key_tail_estimate2}, and use the subadditivity of the
  function $t\mapsto t^p$ to arrive at
  \begin{equation}
    \label{eq:18}
    \E \sup_{x\in[a,1]} \left|\zeta(x)-\zeta_n(x)\right|^p
    \leq \sum_{i=n+1}^{\infty} (\E Q_i)^p\; \E \left(\sup_{x\in[a,1]}
      \prod_{k=1}^{i-1}L_{f_{k,[0,x]}}\right)^p.
  \end{equation}
  Fix $i\geq 2$.  In order to calculate the last expectation, recall
  that $(t_k,x_k,f_k)_{k=1,\dots,i-1}$ is the enumeration of the
  first $i-1$ atoms of $\sP_{[0,1]}$ ordered such that
  $t_1<t_2<\cdots<t_{i-1}$. Let $x_{(i-1:1)}<\cdots <x_{(i-1:i-1)}$ be
  the ordered points $x_1,\dots,x_{i-1}$, and let $f_{(i-1:k),[0,x]}$ be the
  corresponding functions. Note that
  \begin{equation}
    \label{eq:20}
    \prod_{k=1}^{i-1}L_{f_{k,[0,x]}}=\prod_{k=1}^{i-1}L_{f_{(i-1,k),[0,x]}}
    =\prod_{k=1}^{i-1} L_{f_{(i-1,k)}}^{\one_{\{a<x_{(i-1,k)}\leq x\}}}
    \prod_{k=1}^{i-1} L_{f_{(i-1,k)}}^{\one_{\{x_{(i-1,k)}\leq a\}}},
  \end{equation}
  The two factors on the right-hand side of \eqref{eq:20} are
  independent by the Poisson property, and the second factor does not
  depend on $x$. Thus,
  \begin{displaymath}
    \E \left(\sup_{x\in[a,1]}\prod_{k=1}^{i-1}L_{f_{k,[0,x]}}\right)^p
    \leq \E \left (\sup_{x\in [a,1]}\prod_{k=1}^{i-1}
      L_{f_{(i-1,k)}}^{\one_{\{a<x_{(i-1,k)}\leq x\}}}\right)^p\E \left(
      \prod_{k=1}^{i-1} L_{f_{(i-1,k)}}^{\one_{\{x_{(i-1,k)}\leq a\}}} \right)^p.
  \end{displaymath}
  Further, since $(f_k)$ and $(x_k)$ are independent,
  \begin{align*}
    \E \left (\sup_{x\in [a,1]}\prod_{k=1}^{i-1}
      L_{f_{(i-1,k)}}^{\one_{\{a<x_{(i-1,k)}\leq x\}}}\right)^p 
    &= \E \left (\sup_{x\in [a,1]}\prod_{k=1}^{i-1}
      L_{f_{k}}^{\one_{\{a<x_{(i-1,k)}\leq x\}}}\right)^p 
    \leq \E \left(\sup_{i\geq 1}\prod_{k=1}^{i-1}L_{f_k}\right)^p\\
    &\leq \sum_{i=1}^{\infty}\E \left(\prod_{k=1}^{i-1}L_{f_k}\right)^p
    =\sum_{i=1}^{\infty} (\Phi(p))^{i-1}=\frac{1}{1-\Phi(p)}<\infty.
  \end{align*}
  Therefore,
  \begin{multline*}
    \E \left( \sup_{x\in[a,1]} \prod_{k=1}^{i-1}L_{f_{k,[0,x]}}\right)^p \leq \frac{1}{1-\Phi(p)} \E \left(
      \prod_{k=1}^{i-1} L_{f_{(i-1,k)}}^{\one_{\{x_{(i-1,k)}\leq a\}}} \right)^p\\
    = \frac{1}{1-\Phi(p)}\sum_{k=0}^{j-1}\binom{i-1}{k}\Phi(p)^k
    a^k(1-a)^{i-1-k} =\frac{(1-(1-\Phi(p))a)^{i-1}}{1-\Phi(p)},
  \end{multline*}
  where the first equality holds by conditioning on the number of
  points $x_{1},x_2,\dots,x_{i-1}$ which fall in the interval
  $[0,a]$. Summarising, we see that the series on the right-hand side
  of \eqref{eq:18} converges and \eqref{eq:12a} follows.
\end{proof}

\section{Path regularity properties}
\label{sec:cont-prop-zeta}

\subsection{Set-indexed functions}
\label{sec:set-index-funct}

As we have already mentioned, the set function $\zeta$ is not a
measure, yet it is possible to show that it is a.s. continuous from
below and from above.

\begin{proposition}
  \label{thm:continuity-bis}
  Assume that $(A_m)_{m\in\NN}$ is a nondecreasing sequence of sets from
  $\Borel_+(\sX)$ such that $\mu(A_{\infty})<\infty$, where
  $A_{\infty}:=\cup_{m=1}^{\infty}A_m$. Then
  \begin{displaymath}
    \zeta(A_m)\toas \zeta(A_{\infty})\quad\text{as }\; m\to\infty.
  \end{displaymath}
  The same holds for a nonincreasing sequence $(A_m)_{m\in\NN}$ from
  $\Borel_+(\sX)$ such that $0<\mu(A_\infty)$ with
  $A_\infty:=\cap_{m=1}^{\infty}A_m$.
\end{proposition}
\begin{proof}
  Assume that $(A_m)_{m\in\NN}$ is nondecreasing. For fixed $t>0$ we
  can write
  \begin{align*}
    |\zeta(A_{\infty})-\zeta(A_m)|
    & \leq |\zeta(A_{\infty})-\zeta_t(A_{\infty})|
    + |\zeta_t(A_{\infty})-\zeta_t(A_m)|+|\zeta_t(A_m)-\zeta(A_m)|\\
    &\leq |\zeta(A_{\infty})-\zeta_t(A_{\infty})|
    + |\zeta_t(A_{\infty})-\zeta_t(A_m)|
    +\sup_{ m\geq 1}|\zeta_t(A_m)-\zeta(A_m)|,
  \end{align*}
  where $\zeta_t$ was defined by \eqref{eq:1}. Letting $m\to\infty$
  yields
  \begin{displaymath}
    \limsup_{m\to\infty}|\zeta(A_{\infty})-\zeta(A_m)|
    \leq |\zeta(A_{\infty})-\zeta_t(A_{\infty})|
    + \sup_{ m\geq 1}|\zeta_t(A_m)-\zeta(A_m)|,
  \end{displaymath}
  since $N_{A_m}(t)\to N_{A_\infty}(t)$ a.s., and so
  $N_{A_n}(t)=N_{A_\infty}(t)$ a.s. for all sufficiently large
  $m$. Letting $t$ go to infinity and applying
  Corollary~\ref{cor:monotone} yield the desired statement.  The proof
  for nonincreasing sequences is similar.
\end{proof}

\begin{remark}
Proposition~\ref{thm:continuity-bis} also holds in the sense of
$L^p$-convergence  for $p\in\mathcal{I}$, see \eqref{eq:i_def} for the definition of $\mathcal{I}$.
\end{remark}

The recursive equation~\eqref{eq:two-dim-fixed-point} makes it
possible to obtain bounds on the increments of $\zeta$. Let $A\subset
B$ with $A,B\in\Borel_+(\sX)$, and let $U$ be distributed in $B$
according to the normalised $\mu$, see formula
\eqref{eq:distribution_of_U} with $\mathcal{X}=B$. Then
\begin{align*}
  |\zeta(B)-\zeta(A)|
  &\dsim|\left(f(\zeta(B))-f(\zeta(A))\right)\one_{\{U\in A\}}+
  (f(\zeta(A))-\zeta(A))\one_{\{U\in B\setminus A\}}|\\
  &\leq L_f|\zeta(B)-\zeta(A)|
  +|f(\zeta(A))-\zeta(A)|\one_{\{U\in B\setminus A\}},
\end{align*}
with independent $\zeta,f$ and $U$ on the right-hand side.  If $K_f=\E L_f<1$
and \eqref{eq:cond2} holds, then $\zeta(A)$ is integrable by
Proposition~\ref{prop:moments}. Since $\zeta(A)$ and $f(\zeta(A))$
share the same distribution, it holds
\begin{displaymath}
  \E|\zeta(B)-\zeta(A)|
  \leq \frac{\mu(B)-\mu(A)}{(1-K_f)\mu(B)} \E|f(\zeta(A))-\zeta(A)|.
\end{displaymath}
Note that the latter expectation does not depend on $A$, since the
distribution of $\zeta(A)$ does not depend on $A$. A similar estimate
can also be written for $p$-th moment, $p>0$, assuming $\Phi(p)<1$ and
$\E|f(z_0)-z_0|^p<\infty$.

\subsection{Stochastic processes on the half-line}
\label{sec:stoch-proc-half}

In the special case of $\mu$ being the Lebesgue measure on
$\sX=\R_+$, we obtain the following result.

\begin{proposition}
  \label{prop:continuity-process}
  The process $(\zeta(x))_{x>0}$ is a.s.\ continuous at every fixed
  (nonrandom) $x>0$; it has c\`adl\`ag paths and it is not pathwise
  continuous and is not pathwise monotone.
\end{proposition}
\begin{proof}
  Fix arbitrary $x>0$. Let $(x_n)_{n\geq 1}$ and $(x'_n)_{n\geq 1}$ be
  sequences such that $x_n\downarrow x$ and $x'_n\uparrow x$ as
  $n\to\infty$. By Proposition~\ref{thm:continuity-bis}, we obtain
  \begin{displaymath}
    \zeta([0,x_n))\toas \zeta([0,x])\quad\text{and}\quad
    \zeta([0,x'_n))\toas \zeta([0,x))\quad \text{as }\;n\to\infty.
  \end{displaymath}
  Since $\zeta(x-):=\zeta([0,x))=\zeta(x)$~a.s. we deduce the first
  property. 

  The process $\zeta_n(x)$ from \eqref{eq:19} is c\`adl\`ag in $x$ by
  construction, being a composition of c\`{a}dl\`{a}g functions
  $f_{i,[0,x]}$.  The uniform convergence established in
  Corollary~\ref{cor:monotone} yields that $(\zeta(x))_{x>0}$
  has c\`adl\`ag paths a.s.~on $[a,b]$ for any $0<a<b$, being the
  uniform limit of c\`adl\`ag functions.

  To show that it is not pathwise continuous, consider the projection
  $\{(t_i,x_i):i\geq 1\}$ of the Poisson process $\sP$ and take an
  atom\footnote{Such atoms are usually called (lower) records of the
    point process.}  $(t,x)$ such that the rectangle
  $[0,t]\times[0,x]$ does not contain other atoms of
  $\{(t_i,x_i):i\geq 1\}$. Then $\zeta$ is not left-continuous at
  $x$. More precisely, for such (random) $x$ we have
  $\zeta(x)=f(\zeta(x-))$, where $f$ is the mark of the atom at
  $(t,x)$. Since the jump $f(\zeta)-\zeta$ has expectation zero, see
  \eqref{eq:fixed-point}, the process is not pathwise monotone.
\end{proof}

Our next result concerns the total variation of $\zeta$.

\begin{theorem}
  \label{thm:total_variation}
  The total variation of the process $(\zeta(x))_{x>0}$ is
  a.s.~finite on every interval $[a,b]$ with $a>0$.
\end{theorem}
\begin{proof}
  Without loss of generality we may consider an interval $[a,1]$ for a
  fixed $a\in (0,1)$. Fix an arbitrary partition $a=y_0<y_1<y_2<\cdots
  <y_{m}=1$, and let
  \begin{displaymath}
    \tau_j=\tau(y_j,y_{j+1})
    :=\inf\{k\geq 1:x_k\in (y_j,y_{j+1}]\},\quad j=0,1,\dots,m-1,
  \end{displaymath}
  be the index of the first point in $\sP$ such that the second
  coordinate of this point falls in $(y_j,y_{j+1}]$. Note that, for
  every fixed $i\in\NN$,
  \begin{align}
    \label{eq:bound_total_variation_proof1}
    \sum_{j=0}^{m-1}\one_{\{\tau_j\leq i\}}
    =\sum_{j=0}^{m-1}\sum_{k=1}^{i}\one_{\{\tau_j=k\}}
    &\leq \sum_{j=0}^{m-1}\sum_{k=1}^{i}\one_{\{x_k\in
      (y_j,y_{j+1}]\}}\notag \\
    &=\sum_{k=1}^{i}\sum_{j=0}^{m-1}\one_{\{x_k\in (y_j,y_{j+1}]\}}
    =\sum_{k=1}^{i}\one_{\{x_k\in (a,1]\}}\leq i.
  \end{align}
  Let us now consider the increments $\zeta(y_{j+1})-\zeta(y_j)$ for
  $j=0,\dots,m-1$. Write $(f_{i,y})_{i\in\NN}$ for the i.i.d.~copies of the
    function $f_y:=f_{[0,y]}$ from \eqref{eq:f_i_A_definition}. We have
  \begin{align*}
    |\zeta(y_{j+1})-\zeta(y_j)|&\leq |\zeta(y_{j+1})
    -f^{1\uparrow (\tau_j-1)}_{y_j}(z_0)|
    +|\zeta(y_j)-f^{1\uparrow (\tau_j-1)}_{y_j}(z_0)|\\
    &=|\zeta(y_{j+1})-f^{1\uparrow (\tau_j-1)}_{y_{j+1}}(z_0)|
    +|\zeta(y_j)-f^{1\uparrow (\tau_j-1)}_{y_j}(z_0)|,
  \end{align*}
  where the second equality holds because $f_{k,y_{j+1}}=f_{k,y_j}$
  for $k<\tau_j$ by the definition of $\tau_j$. By the
  Lipschitz property, 
\begin{multline*}
    |\zeta(y_{j+1})-\zeta(y_j)| \leq L_{f_{1,y_{j+1}}}\cdots L_{f_{\tau_j-1,y_{j+1}}}
    |f^{\tau_j\uparrow\infty}_{y_{j+1}}(z_0)-z_0|\\
    +L_{f_{1,y_{j}}}\cdots L_{f_{\tau_j-1,y_{j}}}
    |f^{\tau_j\uparrow\infty}_{y_{j}}(z_0) -z_0|,
\end{multline*}
  and, with the help of \eqref{eq:key_tail_estimate},
  \begin{equation}
    \label{eq:7}
    |\zeta(y_{j+1})-\zeta(y_j)|
    \leq \sum_{i=\tau_j}^{\infty}Q_i\prod_{k=1}^{i-1}L_{f_{k,y_{j+1}}}
    +\sum_{i=\tau_j}^{\infty}Q_i\prod_{k=1}^{i-1}L_{f_{k,y_{j}}}
    \leq 2\sum_{i=\tau_j}^{\infty}Q_i\left(\sup_{x\in [a,1]}
      \prod_{k=1}^{i-1}L_{f_{k,x}}\right),
  \end{equation}
  where $Q_i:=|f_i(z_0)-z_0|$ as in the proof of
  Theorem~\ref{thr:uniform}.  Summing over $j=0,\dots,m-1$ and
  subsequently taking the supremum over the set $P$ of all partitions,
  we deduce
  \begin{align}
    \sup_P \sum_{j=0}^{m-1}|\zeta(y_{j+1})-\zeta(y_j)|
    &\leq \sup_P \sum_{i=1}^{\infty}Q_i\left(\sup_{x\in
        [a,1]}\prod_{k=1}^{i-1}L_{f_{k,x}}\right)\sum_{j=0}^{m-1}\one_{\{\tau_j\leq
      i\}}\notag \\
    &\leq \sum_{i=1}^{\infty}iQ_i
    \left(\sup_{x\in [a,1]}\prod_{k=1}^{i-1}L_{f_{k,x}}\right),
  \end{align}
  where \eqref{eq:bound_total_variation_proof1} yields the last
  bound. It remains to note that the series on the right-hand side
  converges a.s. by the Cauchy radical test using the same reasoning
  as in the proof of Theorem~\ref{thr:uniform} in conjunction with a
  trivial observation $i^{1/i}\to1$ as $i\to\infty$. The proof is
  complete.
\end{proof}

Inequality \eqref{eq:7} provides an upper a.s. bound on the increments
of the process $\zeta$ in terms of the tail of a convergent series.

Denote by $V_p(\zeta;[a,b])$ the $p$-variation of $\zeta$ over the
interval $[a,b]$.  The next result demonstrates that the $p$-variation
of $\zeta$ is integrable. Recall that the set $\mathcal{I}$ was
defined in \eqref{eq:i_def}.

\begin{proposition}
  \label{prop:exp-tv}
  For every $p\in\mathcal{I}\cap (0,1]$ and $0<a\leq b$ we have 
  \begin{equation}
    \label{eq:22}
    \E V_p(\zeta;[a,b])
    \leq (\log (b/a))\frac{2\E|Z_{\infty}|^{p}}{1-\Phi(p)}.
  \end{equation}
\end{proposition}
\begin{proof}
  First of all note that $\E|Z_{\infty}|^{p}$ is finite by
  Proposition~\ref{prop:moments}. Further, it suffices to prove the
  statement for the interval $[a/b,1]$ with $0<a<b$. The scale
  invariance property then yields the desired result for the interval
  $[a,b]$. Thus, without loss of generality, assume that $b=1$ and
  $a\in (0,1)$.

  The process $(\zeta(x))_{x\in (0,1]}$ has a countable dense set of
  jumps occurring at points $(x_k)_{k\geq 1}$, where the enumeration
  $(t_k,x_k,f_k)_{k\geq 1}$ of the atoms of $\sP$ is such that
  $0<t_1<t_2<t_3<\cdots$ and $x_k\in [0,1]$ a.s.

  The size of a jump at $x_k\in[0,1]$, $k\in\NN$, depends on the
  number of iterations applied before $f_k$. More precisely,
  \begin{displaymath}
    \zeta(x_k)-\zeta(x_k-)\dsim f^{1\uparrow (\theta_k+1)}(Z_\infty)
    -f^{1\uparrow\theta_k}(Z_\infty), \quad k\in\NN,
  \end{displaymath}
  where $\theta_k:=\sum_{j=1}^{k-1}\one_{\{x_j<x_k\}}$, and $Z_\infty$
  is independent of $(f_j)_{j\leq k}$ and distributed like $\zeta(x)$
  for any $x$. Taking the expectation and using independence we obtain
  \begin{align*}
    \E |\zeta(x_k)-\zeta(x_k-)|^{p}\one_{\{x_k>a\}}
    &\leq \E \left((\Phi(p))^{\theta_k}\one_{\{x_k>a\}}\right)
    \E|f(Z_\infty)-Z_\infty|^{p}\\
    &\leq \E \left((\Phi(p))^{\theta_k}\one_{\{x_k>a\}}\right)
    2\E|Z_\infty|^{p},
  \end{align*}
  where \eqref{eq:fixed-point} has been utilised on the last
  step. Since $x_k$ has the uniform distribution on $[0,1]$, 
  \begin{multline}
    \label{eq:average_variation_calculation1}
    \E\left((\Phi(p))^{\theta_k}\one_{\{x_k>a\}}\right) 
    = \sum_{j=0}^{k-1}\binom{k-1}{j}(\Phi(p))^j \int_a^1 y^j
    (1-y)^{k-1-j}\diff y\\
    =\int_a^1 (1-y+y(\Phi(p)))^{k-1}\diff y
    =\frac{(1-(1-(\Phi(p)))a)^k-(\Phi(p))^k}{k(1-(\Phi(p)))}.
  \end{multline}
  Finally, note that 
  \begin{displaymath}
    \E \sum_{x_k\in [a,1]}|\zeta(x_k)-\zeta(x_k-)|^{p}
    =(-\log a)\frac{2\E|Z_{\infty}|^{p}}{1-\Phi(p)}. \qedhere
  \end{displaymath}
\end{proof}

By the subadditivity of the function $t\mapsto t^p$, Proposition~\ref{prop:exp-tv} implies that the total variation of $\zeta$ over $[a,b]$ is $p$-integrable with a bound on its $p$th moment given by the right-hand side of \eqref{eq:22}.

For each $\eps>0$, the set $J:=\{x>0:\;|\zeta(x)-\zeta(x-)|\geq\eps\}$
of jumps of size at least $\eps$ is a scale invariant point process on
$(0,\infty)$. Indeed, since the $p$-variation of $\zeta$ is
finite on any interval $[a,b]$ with $0<a<b<\infty$ and sufficiently
small $p>0$, the number of points in $J\cap[a,b]$ is a.s.~finite. For $c>0$,
$cJ$ is the functional of $(\zeta(c^{-1}x))_{x>0}$, which coincides in
distribution with $(\zeta(x))_{x>0}$, hence $J$ is scale invariant.

\subsection{Integration with respect to \texorpdfstring{$\zeta$}{zeta}}
\label{sec:integr-with-resp}

Since the process $\zeta$ has finite total variation, it is possible
to integrate continuous functions with respect to $\zeta$ on 
intervals bounded from $0$ in the sense of Riemann--Stieltjes
integration. For arbitrary $0<a<b<\infty$ and continuous $h:[a,b]\to
\R$ we have
\begin{equation}
  \label{eq:integration_with_respect_to_zeta}
  \int_{(a,b]}h(x)\diff \zeta(x)
  =\sum_{x_k\in (a,b]}h(x_k)(\zeta(x_k)-\zeta(x_k-)),
\end{equation}
and the series on the right-hand side converges absolutely a.s.

The following proposition shows that one can also integrate over
intervals $(0,b]$ provided $h$ satisfies an additional integrability
assumption. 

\begin{proposition}
  \label{prop:integrator}
  Let $h:[0,1]\mapsto\R$ be a continuous function such that
  $\int_0^1|h(t)|^{p} t^{-1}\diff t<\infty$ for some
  $p\in\mathcal{I}\cap (0,1]$. Then the limit $\int_{(0,1]}h(x)\diff
  \zeta(x)$ of \eqref{eq:integration_with_respect_to_zeta} as
  $a\downarrow 0$ exists a.s. and in $L^{p}$.
\end{proposition}
\begin{proof}
  It suffices to show that 
  \begin{equation}
    \label{eq:integration_proof_eq1}
    \E \left( \sum_{x_k\in(0,\,1]} |h(x_k)| 
      |\zeta(x_k)-\zeta(x_k-)|\right)^{p}<\infty.
  \end{equation}
  This immediately implies
  \begin{displaymath}
    \sum_{x_k\in(0,\,1]} |h(x_k)| |\zeta(x_k)-\zeta(x_k-)|
    <\infty \quad \text{a.s.}
  \end{displaymath}
  and, thus, by the dominated convergence
  \begin{displaymath}
    \sum_{x_k\in (a,\,1]} h(x_k) (\zeta(x_k)-\zeta(x_k-))
    \toas \sum_{x_k\in (0,\,1]} h(x_k)(\zeta(x_k)-\zeta(x_k-))
    \quad \text{as }\; a\downarrow 0.
  \end{displaymath}
  Furthermore, \eqref{eq:integration_proof_eq1} implies that
  \begin{displaymath}
    \left\|\int_{(a,\,1]}h(x)\diff \zeta(x)-\int_{(0,\,1]}h(x)
      \diff \zeta(x)\right\|_{p}
    \leq \E \left( \sum_{x_k\in(0,\,a]} |h(x_k)|
      |\zeta(x_k)-\zeta(x_k-)|\right)^{p}\to 0
  \end{displaymath}
  as $a\downarrow 0$.  To prove \eqref{eq:integration_proof_eq1}, note
  that subadditivity of $t\mapsto t^{p}$ yields
  \begin{displaymath}
    \E \left( \sum_{x_k\in(0,\,1]} |h(x_k)| |\zeta(x_k)-\zeta(x_k-)|\right)^{p}
    \leq \sum_{k=1}^\infty \E \left(|\zeta(x_k)-\zeta(x_k-)|^{p}
      \cdot |h(x_k)|^{p}\right).
  \end{displaymath}
  Let us prove that the series on the right-hand side converges. Using
  the same calculations as in
  \eqref{eq:average_variation_calculation1}, we derive
  \begin{align*}
    \E \left(|\zeta(x_k)-\zeta(x_k-)|^{p}\cdot |h(x_k)|^{p}\right)
    &\leq 2\E |Z_{\infty}|^{p}
    \E \left((\Phi(p))^{\theta_k}|h(x_k)|^{p}\right)\\
    &=2\E|Z_{\infty}|^{p}\int_0^1 (1-y+y\Phi(p))^{k-1}|h(y)|^{p}\diff y.
  \end{align*}
  The right-hand side is summable to $\frac{2\E
    |Z_{\infty}|^{p}}{1-\Phi(p)}\int_0^1 |h(y)|^{p}y^{-1}\diff y$,
  which is finite by assumptions. Note that $\E|Z_{\infty}|^{p}$ is
  finite by Proposition~\ref{prop:moments}, since $\Phi(p)<1$ and $\E
  |f(z_0)-z_0|^{p}<\infty$ in view of \eqref{eq:cond2}.
\end{proof}

Let us now turn to integration of $\zeta$ with respect to a continuous
deterministic function $h$. The easiest way to define such integrals
is via integration by parts, that is, for $0<a<b$ and continuous
$h:[a,b]\mapsto \R$, put
\begin{align}
  \label{eq:integration_of_zeta_def}
  \int_{(a,\,b]}\zeta(x)\diff h(x)&:=\zeta(b)h(b)-\zeta(a)h(a)
  -\int_{(a,b]}h(x)\diff \zeta(x) \notag \\
  &=\zeta(b)h(b)-\zeta(a)h(a)-\sum_{x_k\in
    (a,\,b]}h(x_k)(\zeta(x_k)-\zeta(x_k-)).
\end{align}
Under the additional integrability assumption, the definition can be
extended to $a=0$.

\begin{theorem}
  \label{thr:integral}
  Under assumptions of Proposition~\ref{prop:integrator} and assuming
  also that $h$ is continuous, the integral
  $\int_{(a,\,b]}\zeta(x)\diff h(x)$ converges in $L^{p}$ as
  $a\downarrow 0$ to a limit which is denoted by
  $\int_{(0,1]}h(x)\diff \zeta(x)$.
\end{theorem}
\begin{proof}
  In view of Proposition~\ref{prop:integrator} and definition
  \eqref{eq:integration_of_zeta_def}, it suffices to check that
  \begin{displaymath}
    \lim_{a\downarrow 0}|h(a)|^{p}\E |\zeta(a)|^{p}=0.
  \end{displaymath}
  But this follows immediately from $\lim_{a\downarrow
    0}|h(a)|=|h(0)|=0$ and $\E |\zeta(a)|^{p}=\E
  |\zeta(1)|^{p}<\infty$, where finiteness is secured by
  Proposition~\ref{prop:moments}.
\end{proof}

In particular, if $\Phi(p)<1$ for some $p\in(0,1]$ then $\int_0^1\zeta(x)\diff
x$ is well defined in $L^p$. Hence, $\int_0^1\zeta(x)\diff
x$ is well defined in $L^p$ for all sufficiently small $p>0$.

\section{Markov property}
\label{sec:markov-property}

Assume that knowledge of the value $y=f(x)$ allows one to recover in a
unique way the value of the argument $x\in\R$ and also the realisation
of the random function $f(\cdot)$ with the distribution $\nu$. To get
a better understanding of this assumption, suppose for a moment that
$\nu$ is supported by a finite set of strictly monotone and strictly
contractive functions $h_1,\dots,h_m\in\sG$ which satisfy the {\it
  strong separation condition}. The latter means that the unique
attractor of the iterated function system $\{h_1,h_2,\dots,h_m\}$,
that is the unique nonempty compact set $\mathcal{K}$ such that
$$
\mathcal{K}=\cup_{i=1}^{m}h_i(\mathcal{K}),
$$
satisfies additionally $h_i(\mathcal{K})\cap
h_j(\mathcal{K})=\varnothing$ for $i\neq j$.  Thus, for every point
$y\in\mathcal{K}$ we can find the unique index $i$ such that $y\in
h_i(\mathcal{K})$, hence, uniquely recover a deterministic function
$h_i$, the realisation of $f$. Since $h_1,\ldots,h_m$ are assumed to
be strictly monotone, it is further possible to find the unique $x\in
\mathcal{K}$ such that $y=h_i(x)=f(x)$. Moreover, since the support of
$Z_{\infty}$, the limit of iterations \eqref{eq:11}, is equal to
$\mathcal{K}$, given the event $\{Z_{\infty}=z\}$ one can uniquely
determine the full (deterministic) sequence
$(g_{n}^{z})_{n\in\NN}\subset \sG$, such that
$$
 z=(g^{z})^{1\uparrow\infty}(z_0).
$$
A typical example of a random Lipschitz mapping satisfying the above
recovery property are Bernoulli convolutions with $\lambda < 1/2$, see
Example~\ref{ex:bernoulli_convolutions} and
Section~\ref{sec:bern-conv} below, for which
$$
m=2,\quad h_1(x)=\lambda x,\quad h_2(x)=\lambda x+1,\quad \nu(\{h_1\})= \nu(\{h_2\})=1/2.
$$
An example where $\nu$ is not finitely supported, yet the
corresponding random function $f$ satisfies the recovery property, is
given by random continued fractions with integer entries, see
Example~\ref{ex:cont-frac} below, in which
$$
f(x)=\frac{1}{\xi+x}
$$
with $\xi\in\NN$ a.s.
From the value $y=f(x)$ one can recover the function $f$ by letting
$\xi$ be the integer part of $1/y$ and $x$ the fractional part of
$1/y$. 

The aim of this section is to show that the process $(\zeta(x))_{x>0}$
generated by a random Lipschitz function $f$ which satisfies the recovery
property, is Markov and calculate the corresponding generators. Denote
by $\salg_{[a,b]}$ the $\sigma$-algebra generated by $\zeta(x)$,
$x\in[a,b]$, where $0<a\leq b\leq \infty$. Write $\salg_a$ for
$\salg_{\{a\}}$.  The recovery property implies that $\salg_a$ is
equal to the $\sigma$-algebra generated by the projection of
$\sP_{[0,a]}$ onto the last component $\sG$. 

\begin{theorem}
  \label{thr:Markov}
  Assume that each $z\in{\rm supp}\,Z_{\infty}$ corresponds to a
  unique sequence $(g_{n}^{z})_{n\in\NN}$ from $\sG$ such that
  \begin{equation}
    \label{eq:23}
    z=(g^{z})^{1\uparrow\infty}(z_0),
  \end{equation} 
  and, for all $n\in\NN$, the mapping $z\mapsto g_{n}^{z}$ is
  measurable as a function from $\R$ to $\sG$. Then the process
  $(\zeta(x))_{x>0}$ is Markov both in forward and reverse time, 
    that is, with respect to filtration $(\salg_{(0,x]})_{x>0}$ and
      $(\salg_{[x,\infty)})_{x>0}$, respectively.
\end{theorem}
\begin{proof}
  Fix $x,u>0$.  Given $\{\zeta(x)=z\}$, we know the projection of
  $\sP_{[0,x]}$ onto $\sG$, which is the sequence
  $(g_{n}^{z})_{n\in\NN}$. The $\sigma$-algebra $\salg_{(0,x]}$ is
  generated by the family of sequences $(g_{n}^{\zeta(y)})_{n\in\NN}$
  with $y\leq x$. Note that $(g_{n}^{\zeta(y)})_{n\in\NN}$ is a
  subsequence of $(g_{n}^{\zeta(y')})_{n\in\NN}$ if $y\leq y'$.

  Let $\kappa_0:=0$ and put
  \begin{displaymath}
    \kappa_{k+1}:=\min\{i > \kappa_k:\; x_i\leq x\},\quad k\geq 0,
  \end{displaymath}
  where $(t_k,x_k,f_k)_{k\in\NN}$ is the enumeration of atoms of
  $\sP_{[0,x+u]}$ such that $(t_k)_{k\in\NN}$ is a.s.~increasing. We
  have
  \begin{equation}
    \label{eq:21}
    \zeta(x+u)=\lim_{n\to\infty} f^{(1)}\circ f_{\kappa_1}\circ 
    \cdots \circ f^{(n)}\circ f_{\kappa_n}(z_0),
  \end{equation}
  where the limit is in the almost sure sense,
  $f_{\kappa_k}=g_{k}^{\zeta(x)}$, and
  \begin{displaymath}
    f^{(k)}(z):=f^{(\kappa_{k-1}+1)\uparrow (\kappa_k-1)}_{(x,x+u]}(z), \quad
    k\in\NN .
  \end{displaymath}
  Note that $\sP_{(x,x+u]}$ is independent of $\sP_{[0,x]}$ by the
  Poisson property and, in particular, $(\kappa_k-\kappa_{k-1})_{k\in\NN}$ are
  i.i.d. with geometric distribution which are also independent of
  $\salg_{(0,x]}$.  Hence, $\zeta(x+u)$ is determined by $\zeta(x)$
  and $\sP_{(x,x+u]}$, so that the conditional distribution of
  $\zeta(x+u)$ given $\salg_{(0,x]}$ coincides with the conditional
  distribution given $\salg_x$.
  
  Let us now prove the Markov property in the reverse time. Given
  $\{\zeta(x)=z\}$ and $y\leq x$, we have
  \begin{displaymath}
    \zeta(y)=(g^{z}_y)^{1\uparrow\infty}(z_0).
  \end{displaymath}
  where $g_{j,y}^{z}$ is equal in distribution to $g^{z}_j$ with
  probability $y/x$, is the identity function with probability $1-y/x$
  and the choices are mutually independent given
  $\{\zeta(x)=z\}$. That is to say, $\zeta(y)$ is a functional of
  $(g_n^{z})_{n\in\NN}$, and $\sP_{[0,x]}$.  Since $\zeta(x)$
  determines the sequence $(g_n^{z})_{k\in\NN}$, and by the
  independence of Poisson processes, we have the Markov property in
  the reverse time.
\end{proof}

The conditional distribution of $\zeta(x+u)$ given $\{\zeta(x)=z\}$
can be determined as follows. Let $(g_n^{z})_{n\in\NN}$ be a sequence
recovered from $\{\zeta(x)=z\}$. In view of \eqref{eq:21},
$\zeta(x+u)$ for $u>0$ can be derived by inserting between each
consecutive pair of functions in the infinite iteration
$$
z=\zeta(x)=g_1^{z}\circ g_2(z)\circ\cdots\circ g_n^{z}\circ\cdots,
$$
an independent copy of a mapping $f^{(k)}$ composed of a geometric
number of independent copies of $f$. The aforementioned geometric
random variables take values in $\{0,1,2,\ldots\}$, are independent
and all have the same parameter $u/(x+u)$. Similarly, it is possible
to determine the conditional distribution of $\zeta(y)$ given
$\{\zeta(x)=z\}$ with $x\geq y$ by deleting each of the functions
$g_n^z$ independently of others with probability $1-y/x$.

Maintaining assumptions of Theorem~\ref{thr:Markov}, we now aim at
finding the generating operator of the time-homogeneous Markov process
$\tilde\zeta(t):=\zeta(e^t)$, $t\in\R$.  This generating operator in
the forward time is defined as the limit
\begin{displaymath}
  (\mathsf{A}_\uparrow h)(z):=\lim_{\delta\downarrow 0}
  \frac{1}{\delta}\Big[\E(h(\zeta(e^{t+\delta}))|\zeta(e^t)=z)-h(z)\Big],
\end{displaymath}
and in the reverse time as 
\begin{displaymath}
  (\mathsf{A}_\downarrow h)(z):=\lim_{\delta\downarrow 0}
  \frac{1}{\delta}\Big[\E(h(\zeta(e^{t-\delta}))|\zeta(e^t)=z)-h(z)\Big]
\end{displaymath}
for all functions $h$ from their domains of definition. 

We calculate the above generators under additional assumptions:
\begin{equation}\label{eq:generator_assumptions}
  L_f\leq c_f\text{ for some deterministic constant } c_f<1
  \quad\text{and}\quad Z_{\infty}\text{ is compactly supported}.
\end{equation}
The above assumption holds, for example, for Bernoulli convolutions. 

\begin{proposition}
  \label{prop:generators}
  Assume that $f$ possesses the recovery property and
  \eqref{eq:generator_assumptions} holds. Then, for $h\in C^1({\rm
    supp}\,Z_{\infty})$, it holds
  \begin{equation}
    \label{eq:24}
    (\mathsf{A}_\uparrow h)(z)=
    \sum_{k=0}^\infty 
    \Big[\E h\big((g^{z})^{1\uparrow k}\circ f 
    \circ ((g^{z})^{1\uparrow k})^{-1}(z)\big)
    - h(z)\Big],\quad z\in{\rm supp}\,Z_{\infty},
  \end{equation}
  and also
  \begin{equation}
    \label{eq:25}
    (\mathsf{A}_\downarrow h)(z)=
    \sum_{k=0}^\infty\Big[h\big((g^{z})^{1\uparrow(k-1)}\circ
    (g_{k}^z)^{-1} 
    \circ ((g^{z})^{1\uparrow(k-1)})^{-1}(z)\big) - h(z)\Big],
    \quad z\in{\rm supp}\,Z_{\infty},
  \end{equation}
  where $(g_j^z)_{j\in\NN}$ is a sequence of deterministic functions
  which is uniquely determined by $z\in{\rm supp}\,Z_{\infty}$.
\end{proposition}

The proof of Proposition \ref{prop:generators} if given in the Appendix.

We close this section by noticing that the Markov property holds
without the recovery property but with respect to a larger filtration
generated by the Poisson process in horizontal strips.

\section{Perpetuities}
\label{sec:perpetuities}

\subsection{Moments and covariances}
\label{sec:moments-covariance}

Let $f(z)=M z+ Q$, where $(M,Q)$ is a random vector in $\R^2$, so that
$L_f=M$. The iterations of i.i.d. copies of the affine random mapping
$z\mapsto M z + Q$ are known as \emph{perpetuities}. In order to avoid
trivialities, we assume throughout this section that
\begin{equation}
  \label{eq:perp_non_degenerate}
  \P\{Mx+Q=x\}<1\quad\text{for all }x\in\R.
\end{equation}
Assume that $\E|M|<1$ and $Q$ is integrable. Then $\zeta(A)$ is
integrable and
\begin{displaymath}
  \E\zeta(A)=\frac{\E Q}{1-\E M}.
\end{displaymath}
If $\E M^2<1$ and $\E Q^2<\infty$, then $\zeta(A)$ is square
integrable for all $A\in\Borel_+(\sX)$, see
\cite[Th.~1.4]{Alsmeyer+Iksanov+Roesler:2008} and
Proposition~\ref{prop:moments}, and the general expression for the
covariance can be found from \eqref{eq:3}. Assuming additionally $\E
Q=0$ and independence of $M$ and $Q$, we obtain
\begin{displaymath}
  \E(\zeta(A_1)\zeta(A_2))
  =\frac{\E Q^2 \mu(A_1\cap A_2)}
  {(1-\E M)\mu(A_2\cup A_1)+(\E M-\E M^2)\mu(A_1\cap A_2)}.
\end{displaymath}
From this we deduce
\begin{displaymath}
  \E\zeta(A)^2 = \frac{\E Q^2}{1-\E M^2},
\end{displaymath}
and 
\begin{displaymath}
  \E(\zeta(A_1)-\zeta(A_2))^2
  =\frac{2\E Q^2(1-\E M)\mu(A_1\triangle A_2)}
  {(1-\E M^2)((\E M-\E M^2)\mu(A_1\cap A_2)
    +(1-\E M)\mu(A_1\cup A_2))}.
\end{displaymath}
Thus, $\zeta$ is continuous in $L^2$ with respect to the convergence
of its argument in measure.  

From now on assume that $\sX=\R_+$, and $\zeta(x)=\zeta([0,x])$. Then
\begin{displaymath}
  \E(\zeta(x)\zeta(y))=\frac{x\E Q^2}
  {(1-\E M)y+(\E M -\E M^2)x},\quad x\leq y.
\end{displaymath}
By exponential change of time, we obtain the stationary process
$\tilde\zeta(s)=\zeta(e^s)$, $s\in\R$, with covariance
\begin{displaymath}
  \E(\tilde\zeta(0)\tilde\zeta(s))=\frac{a}{ce^{|s|}+1},
\end{displaymath}
where $a>0$ and $c>1$. Note that the covariance is not differentiable
at zero, so the process is not $L^2$-differentiable. 

If $M$ and $Q$ are independent, but $Q$ is not centred, then 
\begin{equation}
  \label{eq:16}
  \E(\zeta(x)\zeta(y))=\frac{1}{1-\E M}\cdot \frac{\E Q^2 (1-\E M)+(2\E M-1) (\E Q)^2+(\E Q)^2 (y/x)}
  {(1-\E M)(y/x)+\E M-\E M^2}.
\end{equation}
The covariance between $\zeta(x)$ and $\zeta(y)$ tends to $(\E
Q)^2/(1-\E M)^2$ as $y/x\to\infty$, which, in particular means that {\it correlation} between $\zeta(x)$ and $\zeta(y)$ tends to $0$, as $y/x\to\infty$.

\subsection{The case of a finite interval}
\label{sec:case-finite-interval}

Now consider iterations of $f(z)=Mz+Q$ on the finite interval $(0,1]$
as described in Section~\ref{sec:iter-finite-interv}. Equation
\eqref{eq:2} can be written as
\begin{equation}
  \label{eq:5}
  (\zeta(x))_{x\in(0,1]}\fod ((M\one_{\{U\leq x\}}+\one_{\{U>x\}})\zeta(x) +
  Q\one_{\{U\leq x\}})_{x\in(0,1]}.
\end{equation}
The process $\zeta(x)$ can be also expressed as the a.s.~(pointwise)
convergent functional series
\begin{equation}
  \label{eq:perpetuity-in-strip}
  \zeta(x)=\sum_{n=1}^\infty M'_{1,x}\cdots M'_{n-1,x} Q'_{n,x},\quad x\in(0,1],
\end{equation}
where 
\begin{displaymath}
  M'_{n,x}:=M_n\one_{\{U_n\leq x\}}+\one_{\{U_n>x\}}
  =M_n^{\one_{\{U_n\leq x\}}},\quad
  Q'_{n,x}:=Q_n\one_{\{U_n\leq x\}},\quad n\geq 1,\quad x\in(0,1].
\end{displaymath}

%

If $M=\lambda\in(0,1)$ is fixed,  then 
\begin{equation}\label{eq:perpetuity_with_constant_factor_repres}
  \zeta(x)= \sum_{n=1}^\infty \lambda^{\one_{\{U_1\leq
      x\}}+\cdots+\one_{\{U_{n-1}\leq x\}}} Q_n \one_{\{U_n\leq x\}}
  =:\sum_{n=1}^{\infty}\lambda^{T_{n-1}(x)}Q_n\one_{\{U_n\leq x\}},
\end{equation}
where $T_n(x):=n\widehat{F}_n(x)$, with $\widehat{F}_n(x)$ being the
standard empirical distribution function for the sample
$\{U_1,\dots,U_n\}$. Note that if $Q_1$ is Gaussian, then the sum also
has a Gaussian distribution, that is, the univariate distributions of
$\zeta(x)$ are Gaussian.  Recall that the distribution of $\zeta(x)$
does not depend on $x>0$.

Let us derive an alternative representation for $\zeta(x)$. Put 
$$
S_n(x):=\inf\left\{k\in\NN:\sum_{j=1}^{k}\one_{\{U_j\leq
    x\}}=n\right\}
=\inf\{k\in\NN: T_{k}(x)=n\},\quad n\in\NN,
$$
and note that $\{T_{k-1}(x)=j,U_k\leq x\}=\{S_{j+1}(x)=k\}$ for all
$j\geq 0$ and $k\in\NN$. Thus,
\begin{align*}
\zeta(x)&=\sum_{n=1}^{\infty}\lambda^{T_{n-1}(x)}Q_n\one_{\{U_n\leq x\}}=\sum_{n=1}^{\infty}\sum_{j=0}^{\infty}\lambda^j \one_{\{T_{n-1}(x)=j\}}Q_n\one_{\{U_n\leq x\}}\\
&= \sum_{j=0}^{\infty}\lambda^j \sum_{n=1}^{\infty}Q_n\one\{S_{j+1}(x)=n\}= \sum_{j=0}^{\infty}\lambda^j Q_{S_{j+1}(x)}.
\end{align*}
Summarising we derive the following representation
\begin{equation}\label{eq:representation_neg_binomial}
\zeta(x)=\sum_{j=1}^{\infty}\lambda^{j-1} Q_{S_{j}(x)},\quad x\in(0,1].
\end{equation}
Note that $S_n(x)$ is distributed as a sum of $n$ independent
geometric random variables on $\{1,2,\ldots\}$ with success
probability $x$.

\subsection{Bernoulli convolutions}
\label{sec:bern-conv}

If $M=\lambda\in(0,1)$ and $Q$ takes values $0$ and $1$ with equal
probabilities, then $\zeta(x) =\zeta([0,x])$ is the Bernoulli
convolution for each $x>0$, see
\cite{Erdoes:1939,Erdoes:1940,Solomyak:1995}\footnote{It is
    often alternatively assumed that $Q$ takes values $1$ and
    $-1$.}. By \eqref{eq:16}, 
\begin{equation}
  \label{eq:4}
  \E(\zeta(x)\zeta(y))=\frac{x+y}{4(1-\lambda)^2(y+\lambda x)}, \quad x\leq y. 
\end{equation}

If $\lambda<1/2$, then the distribution of $\zeta(x)$ and the
finite-dimensional distributions of the process $\zeta$ are singular.
If $\lambda=1/2$, then $\zeta(x)$ has the uniform distribution on
$[0,2]$ for all $x$. Let $\mu_{BC,1/2}^{(x)}$ denote the joint
distribution of $(\zeta(x),\zeta(1))$ for $\lambda=1/2$. A sample from
the distribution $\mu_{BC,1/2}^{(0.8)}$ is shown on
Figure~\ref{fig:bernoulli-convolutions}, suggesting that the
$\mu_{BC,1/2}^{(x)}$ is singular for $x\in(0,1)$.

\begin{figure}[!t]
  \label{fig:bernoulli-convolutions}
  \begin{center}
    \includegraphics[scale=0.25]{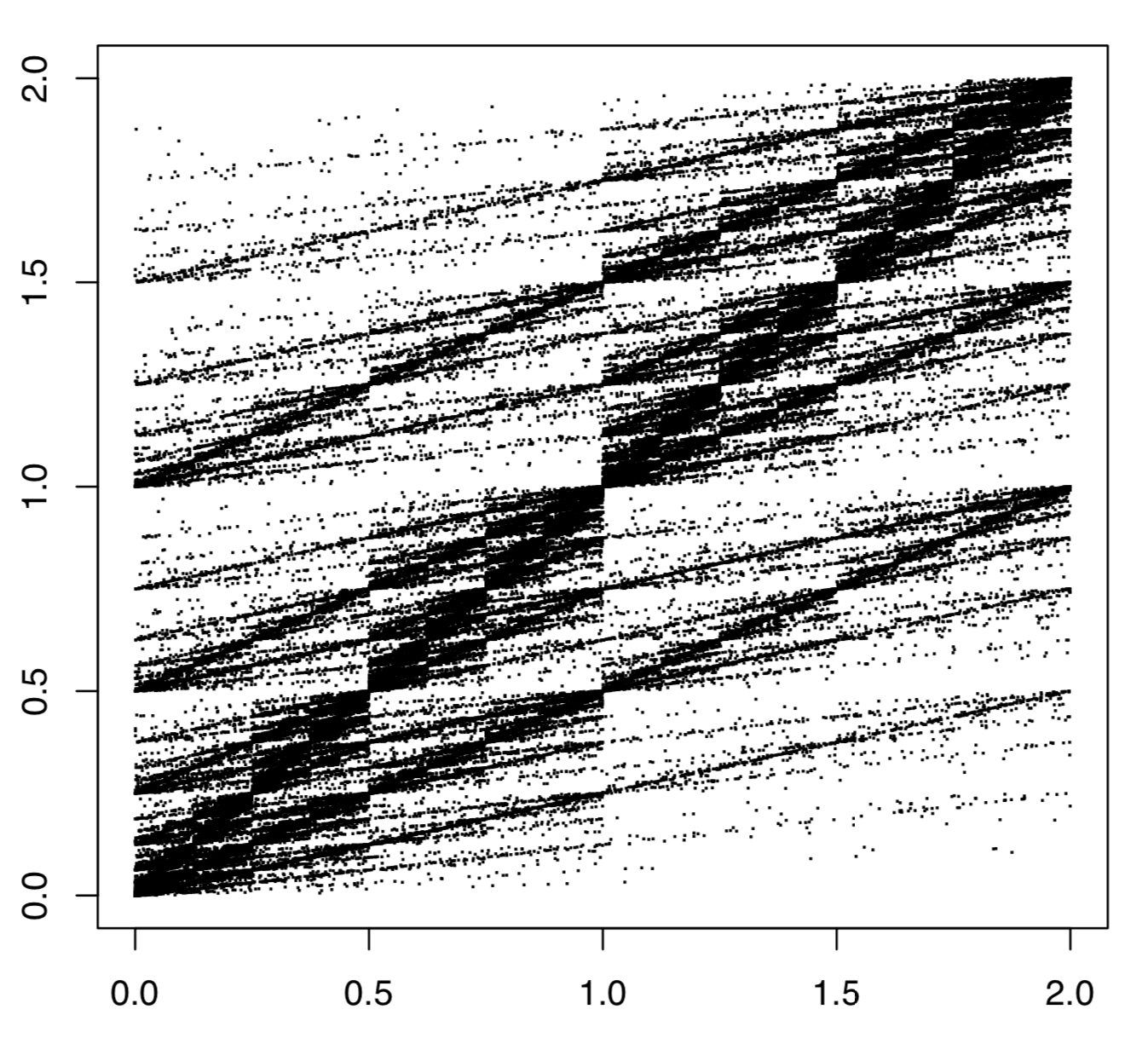}
    \caption{A simulated sample of values for
      $(\zeta(0.8),\zeta(1))$ for Bernoulli convolutions.}
  \end{center}
\end{figure}

The probability measure $\mu_{BC,1/2}^{(x)}$ is the invariant measure for the affine
iterated function system on $\R^2$ generated by $g_i(z):=\bMh_iz+\bQh_i$,
$i=1,\dots,4$, where
\begin{displaymath}
  \bMh_1=\bMh_2:=
  \begin{pmatrix}
    1/2 & 0\\
    0 & 1/2 
  \end{pmatrix},\quad 
  \bMh_3=\bMh_4:=
  \begin{pmatrix}
    1 & 0\\
    0 & 1/2 
  \end{pmatrix},
\end{displaymath}
and 
\begin{displaymath}
  \bQh_1:=(1,1), \; \bQh_2:=(0,0), \; \bQh_3:=(0,1), \bQh_4:=(0,0).
\end{displaymath}
The corresponding probabilities 
\begin{equation}
  \label{eq:29}
  p_1=p_2:=\frac{x}{2},\quad p_3=p_4:=\frac{1-x}{2}
\end{equation}
determine a measure on $\{1,\dots,4\}$ and then the product measure
$m$ on $\{1,\dots,4\}^{\mathbb{Z}}$. Then $\mu_{BC,1/2}^{(x)}$ is the
image of $m$ under the map 
\begin{displaymath}
  \{1,\dots,4\}^{\mathbb{Z}_+}\ni (i_0,i_1,\ldots)\mapsto
  \lim_{n\to\infty}(\bQh_{i_0}+\bMh_{i_0}\bQh_{i_1}
  +\cdots+\bMh_{i_0}\cdots \bMh_{i_{n-1}}\bQh_{i_n}).
\end{displaymath}
The above system of affine maps $g_1,g_2,g_3,g_4$ exhibits exact
overlaps, for example, $g_1\circ g_3\circ g_3\circ g_1=g_3\circ
g_1\circ g_1\circ g_3$.

The top Lyapunov exponent is
\begin{displaymath}
  \lambda_1(x):=\lim_{n\to\infty}\frac{1}{n}
  \log\|\bMh_{i_1}\cdots \bMh_{i_n}\|=-x\log 2,
\end{displaymath}
where the limit holds for $m$-almost all sequences $(i_1,i_2,\ldots)$
by the strong law of large numbers.  Since the top Lyapunov exponent
is negative, the iterated function system is contracting on
average. Noticing that $m$ is ergodic, Theorem 1.2 in \cite{fen19}
applies and yields that $\mu_{BC,1/2}^{(x)}$ is exact dimensional. By definition, 
this means that the limit
\begin{equation}
  \label{eq:17}
  \mathrm{dim}_{\mathrm{loc}}(\mu_{BC,1/2}^{(x)},z)
  :=\lim_{r\downarrow 0} \frac{\log \mu_{BC,1/2}^{(x)}(B_r(z))}{\log r},
\end{equation}
which defines the local dimension of $\mu_{BC,1/2}^{(x)}$ at point $z$, exists and takes the same value for $\mu_{BC,1/2}^{(x)}$-almost all $z$. Moreover, this common value coincides with the Hausdorff dimension $\mathrm{dim}_\mathrm{H}\mu_{BC,1/2}^{(x)}$. Here $B_r(z)$ is the Euclidean ball of radius
$r$ centred at $z$.  The dimension formula of Feng
\cite[Th.~1.3]{fen19} applies in this case and yields that
\begin{equation}
  \label{eq:27}
  \mathrm{dim}_\mathrm{H}\mu_{BC,1/2}^{(x)}=
  \frac{h_1-h_0}{\lambda_1(x)}+\frac{h_2-h_1}{\lambda_2},
\end{equation}
where $\lambda_2=-\log 2$ is the second Lyapunov exponent, and
$h_0,h_1,h_2$ are (conditional) entropies of the system. First, 
\begin{displaymath}
  h_0=-x\log x -(1-x)\log(1-x) +\log 2:= I(x)+\log 2
\end{displaymath}
is the unconditional entropy of the distribution \eqref{eq:29}. While
the exact calculation of $h_2$ constitutes a hard
combinatorial problem, $h_1$ can be determined by noticing that the
first summand in \eqref{eq:27} is equal to the dimension of the
invariant measure for the iterative system on the line composed of the
functions $x/2+1,x/2,x,x$ with probabilities
\eqref{eq:29}\footnote{The authors are grateful to D.-J.~Feng for this
  argument.}. This invariant measure is the uniform distribution on
$[0,2]$, hence, $h_1=h_0-x\log 2$. Finally, since $h_2\geq 0$, we
obtain
\begin{equation}
  \label{eq:28}
  1\leq \mathrm{dim}_\mathrm{H}\mu_{BC,1/2}^{(x)}
  \leq \min(2,2-x+I(x)/\log 2).
\end{equation}
The upper bound alternatively arises from the calculation of the
Lyapunov dimension of the iterative function system, see
\cite{jor:pol:sim07} and \cite{hoc:rap19}.  Note that the right-hand
side of \eqref{eq:28} is smaller than $2$ if and only if $x\in
(x^{\ast},1)$ where $x^{\ast}\approx 0.772908$ is the unique positive
root of the equation $I(x)=x\log 2$. Thus, $\mu_{BC,1/2}^{(x)}$ is
singular for $x>x^{\ast}$. In particular, if $x=0.8$, then the upper
bound equals $\approx1.92$, confirming singularity of the distribution
corresponding to Figure~\ref{fig:bernoulli-convolutions}. We
conjecture that $\mu_{BC,1/2}^{(x)}$ is singular for all $x\in(0,1)$.

\begin{theorem}
  \label{thr:BC-dim}
  The local dimension of the distribution $\mu_{BC,1/2}^{(x)}$ of
  $(\zeta(x),\zeta(1))$ in the Bernoulli convolution scheme with
  $\lambda=1/2$ equals
  \begin{equation}
    \label{eq:26}
    \mathrm{dim}_{\mathrm{loc}}(\mu_{BC,1/2}^{(x)},z)= 2-\frac{\log(1+x)}{\log 2}
  \end{equation}
  for arbitrary $z:=(z_1,z_2)\in [0,2]^2$ with finite binary
  expansions such that the expansion of $z_1$ is a substring of the
  expansion of $z_2$.
\end{theorem}

The proof is postponed to Appendix.

Since the binary rational points in $[0,2]^2$ have
$\mu_{BC,1/2}^{(x)}$-measure zero, Theorem~\ref{thr:BC-dim} does not
allow us to conclude that the dimension of $\mu$ is given by
\eqref{eq:26}. We leave the stronger variant of this statement as a
conjecture. Note that \eqref{eq:26} complies with the bounds given in
\eqref{eq:28}.

\begin{conj}
We conjecture that $\mathrm{dim}_{\mathrm{H}} \mu_{BC,1/2}^{(x)}=2-\frac{\log(1+x)}{\log 2}$.
\end{conj}

The following result shows that the right-hand side of \eqref{eq:26}
provides a lower bound on the dimension of $\mu_{BC,1/2}^{(x)}$.

\begin{theorem}
  \label{thr:BC-dim-bis}
  The dimension of the distribution $\mu_{BC,1/2}^{(x)}$ of
  $(\zeta(x),\zeta(1))$ in the Bernoulli convolution scheme with
  $\lambda=1/2$ satisfies
  \begin{equation}
    \label{eq:26-1}
    \mathrm{dim}_{\mathrm{H}} \mu_{BC,1/2}^{(x)}\geq 2-\frac{\log(1+x)}{\log 2}.
  \end{equation}
\end{theorem}

We close the section on perpetuities by referring the reader to the last subsection of the Appendix where two further examples related to perpetuities are discussed in brief.



\section{Other examples}
\label{sec:other-examples}

\begin{example}
  \label{ex:degenerate}
  Assume that $f(x)\equiv Q$ for some random variable $Q$ whose
  distribution we denote by $\P_Q$. Note that this case corresponds to
  a degenerate perpetuity with $M=0$ a.s. Let us assume that
  $\mathcal{X}$ is $[0,\infty)$ with $\mu$ being the Lebesgue
  measure. The process $\mathcal{P}$ can be regarded as a marked
  Poisson process on $[0,\infty)^2$ with unit intensity and the marks
  being i.i.d. random variables $(Q_k)$ with distribution $\P_Q$ which
  are also independent of positions of the points in
  $\mathcal{P}$. Let $(t_k,x_k)_{k\in\mathbb{Z}}$ be the set of lower
  left records of $\mathcal{P}$ such that
  $(t_0,x_0)$ and $(t_1,x_1)$ are separated by the bisectrix
  $x=t$. Using this notation, the process $(\zeta(x))_{x>0}$ can be
  written as follows
$$
    \zeta(x)=Q_{\inf\{n\in\mathbb{Z}:x_n\leq x\}}, \quad x>0.
$$
In other words, $\zeta(x)=Q_i$ if $x\in(x_{i-1},x_i]$,
$i\in\mathbb{Z}$.  The jump points $(x_k)_{k\in\mathbb{Z}}$ form a
scale invariant Poisson point process, see, for example,
\cite[Prop.~2]{gned08}. After the exponential time change, we obtain a
process $(\tilde{\zeta}(s))_{s\in\R}=(\zeta(e^s))_{s\in\R}$ that takes
i.i.d.~values distributed as $Q$ between the points of a standard
two-sided Poisson process on $\R$ with unit intensity.
\end{example}

\begin{example}
  Assume that $f(z)=\max(1,e^\xi z)$, where $\E\xi<0$. Then
  $L_f=\min(1,e^\xi)$. The backward iterations converge a.s. to a
  random variable $e^Y$ such that $Y$ satisfies the Lindley equation
  $Y\dsim\max(0,\xi+Y)$ from queuing theory. It is well known that
  $Y$ is distributed as
  \begin{displaymath}
    \sup_{j\geq 0} \sum_{i=1}^j \xi_i,
  \end{displaymath}
  where $(\xi_i)_{i\in\NN}$ are i.i.d.~copies of $\xi$.  In other
  words, $\zeta(x)$ is the supremum of a random walk with negative
  drift. For the corresponding process $(\zeta(x))_{x\in(0,1]}$ we
  have the representation
  \begin{displaymath}
    \zeta(x)=\sup_{j\geq 0} \sum_{i=1}^j
      \xi_i\one_{\{U_i\geq x\}},\quad x\in(0,1],
  \end{displaymath}	
  where $(U_i)_{i\in\NN}$ are i.i.d.~uniform on $[0,1]$ which are also
  independent of $(\xi_j)_{j\in\NN}$.
\end{example}

\begin{example}
  \label{ex:cont-frac}
  Let $f(z)=1/(z+\xi)$, where $\xi$ is a positive random variable and
  $z\geq 0$. The iterations produce random continued fractions, see,
  for example, \cite{let:ses83}. The Lipschitz constant of $f$ is
  $L_f=\xi^{-2}$, so \eqref{eq:cond1} and \eqref{eq:cond2} are
  fulfilled if $\E \xi^{-2}<\infty$ and $\E \log\xi>0$.

  If $\xi$ is Gamma distributed, then the
  backwards iterations converge almost surely, and the limit
  $\zeta(x)$ has the inverse Gaussian distribution. Therefore, one
  obtains a stochastic process whose all univariate marginals are inverse
  Gaussian.

  If $\xi$ takes values from $\NN$, then it is possible to uniquely
  recover the sequence of iterations from the limit, so
  Theorem~\ref{thr:Markov} yields the Markov property of the process
  $(\zeta(x))_{x>0}$. 
  
\end{example}


\section{Concluding remarks}
\label{sec:concluding-remarks}

Most of the presented results (with appropriate amendements) hold for
Lipschitz functions taking values in an arbitrary Polish space; in this
case, one obtains set-indexed functions with values in this Polish space.



It is possible to amend the sieving construction in various ways. For
instance, let $\sP$ be the Poisson process $\{(x_i,f_i)\}$ in $\R^d$
marked by i.i.d.~random Lipschitz functions satisfying
\eqref{eq:cond1} and \eqref{eq:cond2}. For each point $x\in\R^d$, order
the points $(x_i)_{i\in\NN}$ of the process according to their
distance to $x$ and take the backward iterations of the corresponding
functions. This results in a random field indexed by $\R^d$ whose
one-dimensional distributions are all identical and which is also
scale invariant.

For yet another alternative construction, let $\sP$ be the Poisson
process $\{(s_i,t_i,f_i)\}$ in $\R\times\R_+\times\sG$. Fix $a>0$, and
for each $x\in\R$ consider the points $(s_i,t_i)$ such that
$|x-s_i|\leq at_i$. Order these points by increasing second
coordinates $t_i$ and let $\zeta(x)$ be the limit of the backwards
iterations of the corresponding functions.

Finally, let us make a concluding remark that a different notion 
of probabilistic sieving related to so-called generalized leader-election procedures
has been recently considered in \cite{Alsmeyer+Kabluchko+Marynych:2016} and \cite{Alsmeyer+Kabluchko+Marynych:2017}.


\section*{Acknowledgements}

The authors have benefited from discussions with a number of
colleagues, especially with Gerold Alsmeyer, Bojan Basrak, Andreas
Basse-O'Connor, Bernardo D'Auria, Lutz D\"umbgen, De-Jun Feng, Sergei
Foss, Alexander Iksanov, Zakhar Kabluchko, Sebastian Mentemeier and
Sid Resnick.  IM was supported by the Swiss National Science
Foundation Grants No 200021\_175584 and IZHRZ0\_180549 and by the
Alexander von Humboldt Foundation. We thank two anonymous referees 
for a number of valuable suggestions, pointers to literature and for detecting 
several oversights in the earlier versions of the paper.

\bibliographystyle{abbrv}
\bibliography{AMIM}

\newpage

\section*{Appendix}

In the Appendix we collect the promised proofs and examples skipped in the main text.

\subsection*{Proof of Proposition \ref{prop:generators}}

  Let us prove \eqref{eq:24}, the proof of \eqref{eq:25} is
  similar. Order all points of $\sP_{[0,x+\delta]}$ according to their
  arrival times $t_i$. Formula \eqref{eq:23} implies that, conditionally
  on $\{\zeta(e^t)=z\}$, the random variable $\zeta(e^{t+\delta})$ is
  distributed as the following a.s.~limit
$$
  \lim_{n\to\infty} (g^{z})^{1\uparrow (\tau_1-1)}\circ f_1
  \circ (g^{z})^{\tau_1\uparrow (\tau_2-1)}\circ
  f_2\circ\cdots\circ (g^{z})^{\tau_{n-1}\uparrow (\tau_n-1)}\circ f_n (z_0),
$$
where 
$$
\tau_0:=0,\quad \tau_n:=\min\{i>\tau_{n-1}:\; x_i\in(e^t,e^{t+\delta}]\},\quad n\in\NN,
$$
and $(f_k)_{k\in\NN}$ are i.i.d. copies of $f$ which are independent
of everything else. Note that $(\tau_n-\tau_{n-1})_{n\in\NN}$ are
i.i.d. and
$$
\P\{\tau_1=j\}=e^{-\delta(j-1)}(1-e^{-\delta}),\quad j\in\NN.
$$  
On the first step we show that it is possible to neglect
$f_2,f_3,\ldots$, which have been inserted after $g_{\tau_1}^z$, that
is
\begin{multline}\label{eq:generator_proof1}
\lim_{\delta\downarrow 0}\frac{1}{\delta}\E \Big(h\Big((g^{z})^{1\uparrow (\tau_1-1)}\circ f_1\circ (g^{z})^{\tau_1\uparrow (\tau_2-1)}\circ f_2\circ\cdots\circ (g^{z})^{\tau_{n-1}\uparrow (\tau_n-1)}\circ f_n \circ \cdots(z_0)\Big)\\
-h\Big((g^{z})^{1\uparrow (\tau_1-1)}\circ f_1\circ(g^{z})^{\tau_1\uparrow \infty}(z_0)\Big)\Big)=0.
\end{multline}
Using the fact that $h^{\prime}$ is continuous, hence, bounded on the
compact set ${\rm supp}\,Z_{\infty}$, we derive using the mean value
theorem for differentiable functions
\begin{multline*}
  \Big|h\Big((g^{z})^{1\uparrow (\tau_1-1)}\circ f_1\circ (g^{z})^{\tau_1\uparrow (\tau_2-1)}\circ f_2\circ\cdots\circ (g^{z})^{\tau_{n-1}\uparrow (\tau_n-1)}\circ f_n \circ \cdots(z_0)\Big)\\
-h\Big((g^{z})^{1\uparrow (\tau_1-1)}\circ f_1\circ(g^{z})^{\tau_1\uparrow \infty}(z_0)\Big|
\leq {\rm const}\cdot c_f^{\tau_2}.
\end{multline*}
Since $\delta^{-1}\E c_f^{\tau_2}=\delta^{-1}(\E c_f^{\tau_1})^2 \to 0$ as $\delta\downarrow0$, the generating operator is given by 
\begin{align}
  (\mathsf{A}_\uparrow h)(z)&=\lim_{\delta\downarrow 0}
  \frac{1}{\delta}\Big[\E h\Big((g^{z})^{1\uparrow (\tau_1-1)}\circ f_1\circ(g^{z})^{\tau_1\uparrow \infty}(z_0)\Big)
  - h(z)\Big]\nonumber\\
  &=\lim_{\delta\downarrow 0}
  \frac{1}{\delta}\Big[\E h\Big((g^{z})^{1\uparrow (\tau_1-1)}\circ f\circ ((g^z)^{1\uparrow (\tau_1-1)})^{-1}(z)\Big)
  - h(z)\Big]\nonumber\\
  &=\lim_{\delta\downarrow 0}
  \frac{1}{\delta}\sum_{k=0}^{\infty}\Big[\E h\Big((g^{z})^{1\uparrow k}\circ f\circ ((g^z)^{1\uparrow k})^{-1}(z)\Big)
  - h(z)\Big]e^{-\delta k}(1-e^{-\delta}),\label{eq:generator_proof2}
\end{align}
where for the second equality we have used that 
\begin{displaymath}
  (g^{z})^{\tau_1\uparrow \infty} (z_0)
  =((g^z)^{1\uparrow (\tau_1-1)})^{-1}(z).
\end{displaymath}
Using the inequality
$$
\E h\Big((g^{z})^{1\uparrow k}\circ f\circ ((g^z)^{1\uparrow k})^{-1}(z)\Big)
  - h(z)\leq {\rm const}\cdot c_f^k,\quad k\geq 0,
$$
by the Lebesgue dominated convergence theorem we can swap the sum and
the limit on the right-hand side of \eqref{eq:generator_proof2}. This
completes the proof of \eqref{eq:24}.

\subsection*{Markov processes generated by Bernoulli convolutions}
\label{sec:mark-proc-gener-1}

As we have mentioned in Section~\ref{sec:markov-property}, the
process $(\zeta(x))_{x>0}$ generated by the mapping $f(x)=\lambda x+Q$
with $\lambda\in(0,1/2)$ and $Q$ equally likely taking the values $0$
and $1$, is Markov both in forward and reverse time.

In order to calculate its generating operator, note that each
$z\in{\rm supp}\,Z_\infty\subset [0,(1-\lambda)^{-1}]$ corresponds to
a sequence $(q_{n}^{z})_{n\in\NN}$ from $\{0,1\}^\NN$ such that
$g_n^z(x)=\lambda x+q_n^{z}$ and
\begin{equation}
  \label{eq:bernoulli_convolutions_generators}
  z=(g^z)^{1\uparrow \infty}(z_0)=\sum_{k=1}^{\infty}\lambda^{k-1}q_k^z.
\end{equation}
Direct calculations yield
\begin{displaymath}
  g_{1}^{z}\circ\cdots\circ g_{k}^{z}\circ f
  \circ (g_{k}^z)^{-1}\cdots (g_{1}^z)^{-1}(x)
  =(1-\lambda)\sum_{i=1}^k \lambda^{i-1}q_i^z+\lambda^k Q +\lambda x,\quad x\in\R.
\end{displaymath}
By \eqref{eq:24},
\begin{displaymath}
  (\mathsf{A}_\uparrow h)(z)=
  \sum_{k=0}^\infty
  \Big[\E h(-(1-\lambda)\hat{z}_k +\lambda^kQ
  +z)-h(z)\Big],
\end{displaymath}
where $Q$ equally likely takes values $0,1$ and 
\begin{displaymath}
  \hat{z}_k:= \sum_{i=k+1}^\infty \lambda^{i-1}q_i^z,\quad k\geq 0.
\end{displaymath}
If $h(z)=z$, then 
\begin{displaymath}
  (\mathsf{A}_\uparrow h)(z)=
  -(1-\lambda)\sum_{k=0}^\infty
  \hat{z}_k+\frac{\E Q}{1-\lambda}
  =-(1-\lambda)\sum_{i=1}^\infty i\lambda^{i-1} q_i^z+\frac{\E Q}{1-\lambda}.
\end{displaymath}
A curious observation is that the sum in the last formula is the
derivative of $\lambda\mapsto \lambda z(\lambda)$ in
\eqref{eq:bernoulli_convolutions_generators}.  The generating operator
in the reverse time is given by
\begin{displaymath}
  (\mathsf{A}_\downarrow h)(z)=\sum_{k=0}^\infty
  \Big[h(-(1-\lambda^{-1})\hat{z}_{k-1} -\lambda^{k-2}q_k^z
  +z)-h(z)\Big].
\end{displaymath}

\subsection*{Proof of Theorems \ref{thr:BC-dim} and \ref{thr:BC-dim-bis}}

\begin{proof}
  First of all, note that we may replace the Euclidean ball in
  \eqref{eq:17} with the $\ell_\infty$-ball. Further, for any sequence
  $r_n\downarrow 0$, there exist a sequence $(k_n)\uparrow\infty$ of
  integers such that $2^{-k_n}\leq r_n< 2^{-k_n+1}$, and by the
  standard sandwich argument we see that it suffices to prove
  \eqref{eq:17} along the sequence $r_k=2^{-k}$ as $k\to\infty$.
  
  Order the points $(t_i,x_i,f_i)_{i\in\NN}$ of the point process
  $\sP_{[0,1]}$ so that $t_1\leq t_2\leq \cdots$. Recall the notation
  \begin{displaymath}
    T_n(x)=\sum_{j=1}^{n}\one_{\{x_j\leq x\}}\quad\text{and}\quad S_n(x)=\inf\{k\in\NN: T_k(x)=n\}, \quad n\in\NN.
  \end{displaymath}
  Since $x$ is fixed, in the following the argument $x$ is omitted. 

  Let $z_1=\sum_{k=1}^{m}\gamma_k/2^{k-1}$, $\gamma_m=1$, and
  $z_2=\sum_{k=1}^{n}\gamma'_k/2^{k-1}$, $\gamma_n^{\prime}=1$, be the
  binary expansions of $z_1$ and $z_2$, respectively. By the
  assumption, $\gamma_1\gamma_2\ldots\gamma_m$ is a substring of
  $\gamma'_1\gamma'_2\ldots\gamma'_n$ and, in particular, $m\leq
  n$. Recalling the representation
  \eqref{eq:representation_neg_binomial} for $\zeta$, we can write
  \begin{displaymath}
    \zeta(x)=\sum_{n=1}^{\infty}\frac{Q_{S_n}}{2^{n-1}}
    \quad\text{and}\quad
    \zeta(1)=\sum_{n=1}^{\infty}\frac{Q_{n}}{2^{n-1}}.
  \end{displaymath}
  For $k>\max(n,m)$, we have
  \begin{align*}
    \mu_{BC,1/2}^{(x)}\left(z+[0,2^{-k}]^2\right)
    &=\Prob{\zeta(x)\in [z_1,z_1+2^{-k}],\zeta(1)\in [z_2,z_2+2^{-k}]}\\
    &=\P\big\{Q_1=\gamma'_1,\dots,Q_n=\gamma'_n,Q_{n+1}=\cdots=Q_{k+1}=0,\\
    &\quad\quad\;\; Q_{S_1}=\gamma_1,\dots,Q_{S_m}=\gamma_m,
    Q_{S_{m+1}}=\cdots=Q_{S_{k+1}}=0\big\}.
  \end{align*}
  Denote the event under the last probability sign by $A$. Since we
  assume $\gamma_m=1$, event $A$ can occur only if
  $\{S_m\leq n\}\cup \{S_{m}>k+1\}$. We proceed by bounding
  $\Prob{A,S_{m}>k+1}$ as follows:
  \begin{align*}
    &\P\{A, S_{m}>k+1\}\\
    &\leq \Prob{S_{m}>k+1,Q_{n+1}=\cdots=Q_{k+1}=0,Q_{S_{m+1}}=\cdots=Q_{S_{k+1}}=0}\\
    &= \Prob{S_{m}>k+1}\left(\frac{1}{2}\right)^{k-n+1}
    \left(\frac{1}{2}\right)^{k-m+1}
    \leq m\Prob{S_1>\frac{k+1}{m}}\left(\frac{1}{2}\right)^{k-n+1}
    \left(\frac{1}{2}\right)^{k-m+1}\\
    &=m(1-x)^{(k+1)/m}\left(\frac{1}{2}\right)^{k-n+1}
    \left(\frac{1}{2}\right)^{k-m+1}
    =\mathcal{O}\left(\frac{(1-x)^{1/m}}{4}\right)^k
    \quad \text{as }\; k\to\infty.
  \end{align*}
  In order to calculate $\Prob{A,S_{m}\leq n}$, note that 
  $\{S_m\leq n\}=\{T_n\geq m\}$. By definition, $S_{T_n}\leq n$ and
  $S_{T_n+1}> n$. Therefore,
  \begin{align*}
    \Prob{A,S_{m}\leq n}
    =\sum_{l=m}^n\sum_{j=l}^n
    \Prob{S_m\leq n,T_n=l,S_l=j,S_{l+1}>n,B_{m,n},C_{n,k}}
  \end{align*}
  where 
  \begin{align*}
    B_{m,n}&=\big\{Q_1=\gamma'_1,\dots,Q_n=\gamma'_n,
    Q_{S_1}=\gamma_1,\dots,Q_{S_m}=\gamma_m,
    Q_{S_{m+1}}=\cdots=Q_{S_{T_n}}=0\big\},\\
    C_{n,k}&=\big\{Q_{n+1}=\cdots=Q_{k+1}=0,Q_{S_{T_n+1}}=\cdots=Q_{S_{k+1}}=0\big\}. 
  \end{align*}
  Note that $S_{l+i}=S_l+S'_i$, $i\geq1$, where $(S'_i)_{i\in\NN}$ is
  a distributional copy of the random walk $(S_i)_{i\in\NN}$. Then
  \begin{multline*}
    \Prob{S_m\leq n,T_n=l,S_l=j,S_{l+1}>n,B_{m,n},C_{n,k}}\\
    =\Prob{C_{n,k},S_{l+1}>n|T_n=l,S_l=j,B_{m,n}}
    \Prob{S_m\leq n,T_n=l,S_l=j,B_{m,n}},
  \end{multline*}
 and further
  \begin{align*}
    &\hspace{-1cm}\Prob{C_{n,k},S_{l+1}>n|T_n=l,S_l=j,B_{m,n}}\\
    &=\Prob{Q_{n+1}=\cdots=Q_{k+1}=0,Q_{j+S'_1}=\cdots=Q_{j+S'_{k-l+1}}=0,j+S'_1>n}\\
    &=\Prob{Q_{n+1}=\cdots=Q_{k+1}=0,Q_{n+S'_1}=\cdots=Q_{n+S'_{k-l+1}}=0}
    \Prob{j+S'_1>n},
  \end{align*}
  where the last equality relies on the memoryless property of the
  geometrically distributed $S'_1$.  Let $N$ be binomially distributed
  $\mathsf{Bin}(k-n+1,x)$.  Then
  \begin{align*}
    \P\big\{Q_{n+1}=&\cdots=Q_{k+1}=0,Q_{n+S'_1}=\cdots=Q_{n+S'_{k-l+1}}=0\big\}\\
    &=\left(\frac{1}{2}\right)^{k-n+1}
    \E \left(\frac{1}{2}\right)^{k-l+1-N}\\
    &=\left(\frac{1}{2}\right)^{k-n+1}
    \sum_{i=0}^{k-n+1} \binom{k-n+1}{i}x^i(1-x)^{k-n+1-i}
    \left(\frac{1}{2}\right)^{k-l+1-i}\\
    &=\left(\frac{1}{2}\right)^{k-l+1}
    \left(\frac{1+x}{2}\right)^{k-n+1}.
  \end{align*}
  Thus, 
  \begin{multline*}
    \Prob{A,S_{m}\leq n}
    =  \left(\frac{1}{2}\right)^{k}
    \left(\frac{1+x}{2}\right)^{k}\\
    \sum_{l=m}^n\sum_{j=l}^n (1-x)^{n-j}
    \left(\frac{1}{2}\right)^{-l+1}\left(\frac{1+x}{2}\right)^{-n+1}
    \Prob{S_m\leq n,T_n=l,S_l=j,S_{l+1}>n,B_{m,n}}.
  \end{multline*}
  Note that the double sum does not depend on $k$. Hence, 
  \begin{displaymath}
    \mu_{BC,1/2}^{(x)}\left(z+[0,2^{-k}]^2\right)
    ={\rm const} \left(\frac{1}{2}\right)^{k}
    \left(\frac{1+x}{2}\right)^{k}
    +\mathcal{O}\left(\frac{(1-x)^{1/m}}{4}\right)^k,
  \end{displaymath}
  where the constant does not depend on $k$ (but might depend on $m$ and $n$). The same expression holds for $\mu_{BC,1/2}^{(x)}\left(z+[-2^{-k},0]^2\right)$. 

  Furthermore,
  \begin{align*}
    \mu_{BC,1/2}^{(x)}&\left(z+[-2^{-k},0]\times[0,2^{-k}]\right)\\
    &=\P\big\{Q_1=\gamma'_1,\dots,Q_n=\gamma'_n,
    Q_{n+1}=\cdots=Q_{k+1}=0,\\
    &\quad\quad\;\;  Q_{S_1}=\gamma_1,\dots,Q_{S_m-1}=\gamma_{m-1},Q_{S_m}=0,
      Q_{S_{m+1}}=\cdots=Q_{S_{k+1}}=1\big\}.
  \end{align*}
  The event under probability sign occurs only if $S_{m+1}\leq n$ or
  $S_{m+1}>k+1$. By taking the double sum over $T_n=l$ and $S_l=j$ for
  $m+1\leq l\leq n$ and $l\leq j\leq n$ as above, we arrive at 
  \begin{displaymath}
    \mu_{BC,1/2}^{(x)}\left(z+[-2^{-k},0]\times[0,2^{-k}]\right)
    \leq {\rm const}\cdot \left(\frac{1-x}{4}\right)^k,
  \end{displaymath}
  where the constant does not depend on $k$. Furthermore, 
  $\mu_{BC,1/2}^{(x)}\left(z+[0,-2^{-k}]\times[-2^{-k},0]\right)$ is
  bounded by the same expression. Thus,
  \begin{displaymath}
    \mu_{BC,1/2}^{(x)}\left(z+[-2^{-k},2^{-k}]^2\right)
    =c \left(\frac{1}{2}\right)^{k}
    \left(\frac{1+x}{2}\right)^{k}
    +\mathcal{O}\left(\frac{(1-x)^{1/m}}{4}\right)^k.
  \end{displaymath}
  Finally, \eqref{eq:17} yields \eqref{eq:26}.   
\end{proof}

\begin{proof}[Proof of Theorem \ref{thr:BC-dim-bis}]
  %
  All points $z:=(z_1,z_2)$ in the support of $\mu_{BC,1/2}^{(x)}$ can
  be represented as binary expansions
  $z_1=\sum_{k=1}^{\infty}\gamma_k/2^{k-1}$ and
  $z_2=\sum_{k=1}^{\infty}\gamma'_k/2^{k-1}$, where the sequences
  $\gamma:=(\gamma_k)_{k\in\NN}$ and $\gamma'=(\gamma'_n)_{n\in\NN}$
  in $\{0,1\}^\NN$ are such that $\gamma$ is a subsequence of
  $\gamma'$. For almost all $z$, there is an infinite increasing
  sequence $(\tau_k)_{k\in\NN}$ of natural numbers such that
  $\gamma_{\tau_k+1}=\gamma'_{\tau_k+1}=0$ and
  $\gamma_{\tau_k+2}=\gamma'_{\tau_k+2}=1$ for all $k\geq1$. This
  follows from the Borel--Cantelli lemma applied to the sequence of
  independent events 
  \begin{displaymath}
    B_n:=\{Q_{Y_n}=0,Q_{Y_n+1}=1,Q_{S_{Y_n}}=0,Q_{S_{Y_n+1}}=1\},\quad
    n\geq 1,
  \end{displaymath}
  where $Y_1=1$, and $Y_{n+1}=S_{Y_n+1}+1$, $n\geq1$.  Note that the
  sequence $(\tau_k)_{k\in\NN}$ is not random, it is determined by the
  sequences $\gamma$ and $\gamma'$.  Given that $\mu_{BC,1/2}^{(x)}$
  is exact dimensional and the limit in \eqref{eq:17} exists, it is
  possible to take the limit along $r_k=2^{-\tau_k}$, $k\in\NN$.

  Consider $\tilde{z}:=(\tilde{z}_1 ,\tilde{z}_2)$ with
  \begin{displaymath}
    \tilde{z}_1:=\sum_{j=1}^{\tau_k} 2^{-(j-1)} \gamma_j
    \quad \text{and}\quad  
    \tilde{z}_2:=\sum_{j=1}^{\tau_k} 2^{-(j-1)} \gamma'_j.
  \end{displaymath}
  Then 
  \begin{displaymath}
    z+[-2^{-(\tau_k+1)},2^{-(\tau_k+1)}]\subset 
    \tilde{z}+[0,2^{-\tau_k}]\subset 
    z+[-2^{-\tau_k},2^{-\tau_k}],
  \end{displaymath}
  where we used that $\gamma_{\tau_k+2}=\gamma'_{\tau_k+2}=1$. 
  Therefore, it suffices to consider 
  \begin{multline*}
    \mu_{BC,1/2}^{(x)}(\tilde{z}+[0,2^{-\tau_k}])\\
    =\P\big\{Q_i=\gamma'_i, i=1,\dots,\tau_k,\;    
      \gamma'_{S_j}=\gamma_j,j=1,\dots,T_{\tau_k},
      Q_{S_l}=\gamma_{S_l},l=T_{\tau_k}+1,\dots,\tau_k\big\}.
  \end{multline*}
  Note that $T_{\tau_k} :=\sum_{j=1}^{\tau_k}\one_{\{x_j\leq x\}}$ has the
  binomial distribution $\mathsf{Bin}(\tau_k,x)$, so that
  \begin{multline*}
    \mu_{BC,1/2}^{(x)}(\tilde{z}+[0,2^{-\tau_k}])
    \leq \P\big\{Q_i=\gamma'_i, i=1,\dots,\tau_k,
      Q_{S_l}=\gamma_{S_l},l=T_{\tau_k}+1,\dots,\tau_k\big\}\\
    =\left(\frac{1}{2}\right)^{\tau_k}
    \E\left(\frac{1}{2}\right)^{\tau_k-T_{\tau_k}}
    =\left(\frac{1}{2}\right)^{\tau_k}\left(\frac{1+x}{2}\right)^{\tau_k}.
  \end{multline*}
  The conclusion follows from \eqref{eq:17}. 
\end{proof}

\subsection*{Further examples related to perpetuities}
\label{sec:other-perpetuities}

\begin{example}
  Assume that $Q$ is standard normal and $M=\lambda\in(0,1)$ is
  constant. Then $\zeta(x)$, $x>0$, has univariate Gaussian marginals,
  and its covariance is given by 
$$
\E (\zeta(x)\zeta(y))=\frac{x}{(1-\lambda)(y+\lambda x)}.
$$  
 By time change $x=e^s$,
  we arrive at a centred stationary process $\tilde\zeta(s)$,
  $s\in\R$, with univariate Gaussian marginals and covariance
  \begin{displaymath}
    \E(\tilde\zeta(0)\tilde\zeta(s))=\frac{1}{(1-\lambda)(e^{|s|} +\lambda)}. 
  \end{displaymath}
  The bivariate distributions of this process are no longer Gaussian,
  see Figure~\ref{fig:normal}.
 
\end{example}

\begin{figure}[!t]\label{fig:normal}
    \begin{center}
      \includegraphics[scale=0.25]{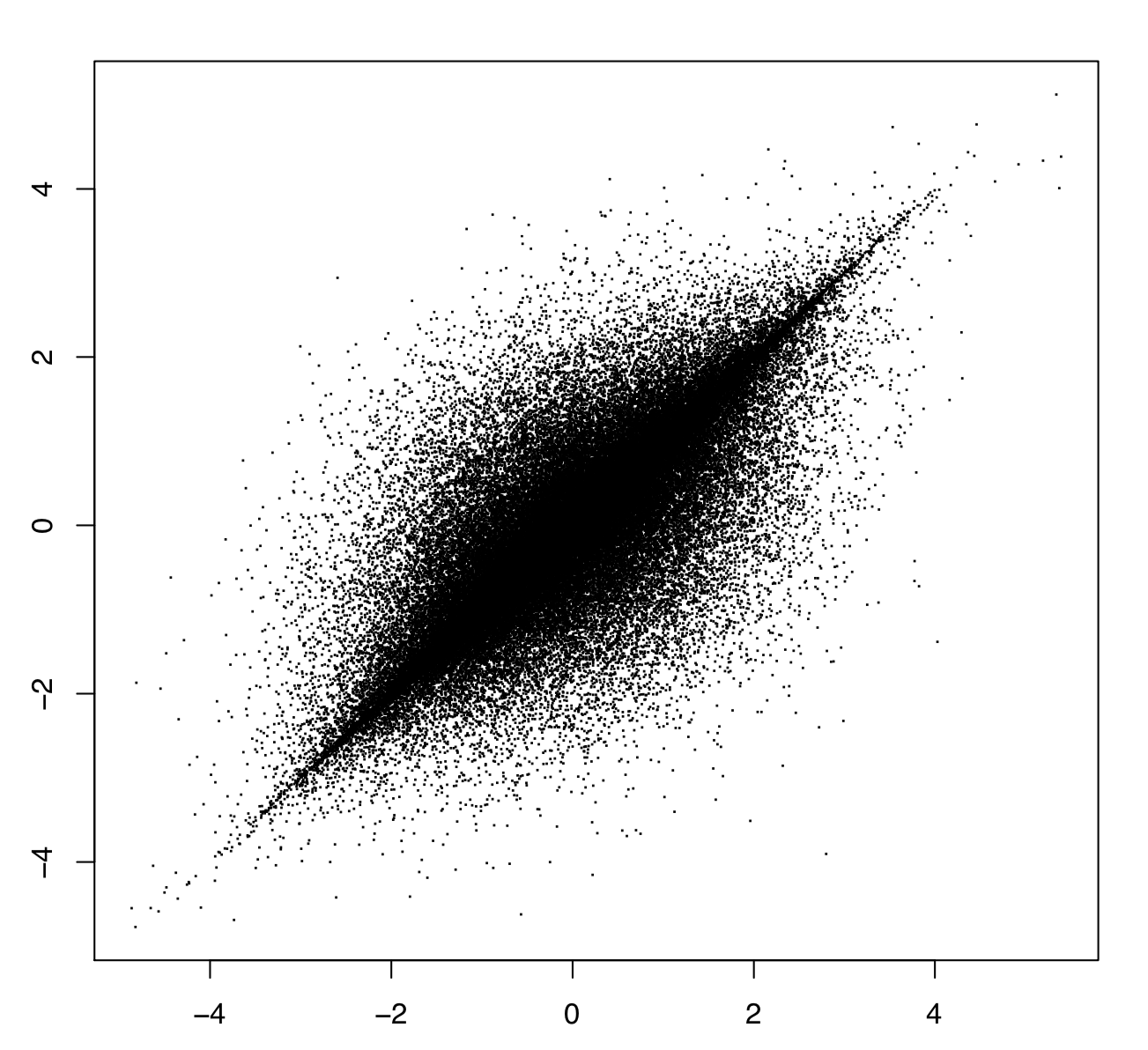}
      \caption{A simulation of $(\zeta(0.7),\zeta(1))$ for normal $Q$
        and $\lambda=1/2$.}
    \end{center}
  \end{figure}	

\begin{example}
  Let $M=Q$ for the standard uniform $Q$. In this case, $\zeta(x)$ for
  each $x>0$ follows the Dickman distribution, see
  e.g. \cite{pen:wad04}. While the obtained stochastic process has all
  univariate Dickman marginals, it does not have independent
  increments like the Dickman process constructed using the infinite
  divisibility property of the Dickman law, see \cite{cov09}.  In our
  case we have the representation
  \begin{displaymath}
    \zeta(x)=\sum_{n=1}^\infty Q_1^{\one_{\{U_1\leq x\}}}\cdots
    Q_{n-1}^{\one_{\{U_{n-1}\leq x\}}} Q_n\one_{\{U_n\leq x\}}.
  \end{displaymath}
\end{example}

\end{document}